\documentclass{article}
\pdfoutput=1
\usepackage[colorlinks,linkcolor=red,citecolor=blue, pagebackref=true]{hyperref}
\usepackage[margin=1.2in]{geometry}
\usepackage[longnamesfirst,numbers]{natbib}
\usepackage{booktabs}

\usepackage{amsthm}

\usepackage{thmtools}
\declaretheorem[name=Theorem,numberwithin=section]{theorem}
\declaretheorem[name=Proposition,sibling=theorem]{proposition}
\declaretheorem[name=Lemma,sibling=theorem]{lemma}

\declaretheorem[name=Remark,sibling=theorem]{remark}
\declaretheorem[name=Definition,sibling=theorem]{definition}

\usepackage{amsmath, amssymb, amsfonts}
\usepackage{mathabx}

\usepackage[utf8]{inputenc}
\usepackage{dsfont}
\usepackage[makeroom]{cancel}
\usepackage[shortlabels]{enumitem}
\usepackage{thmtools}
\usepackage[capitalise]{cleveref}
\usepackage{accents}
\usepackage[normalem]{ulem}
\usepackage{tcolorbox}
\usepackage{autonum}

\usepackage{colortbl}

\newcommand{\PP}{\mathbb P}
\newcommand{\RR}{\mathbb R}
\newcommand{\EE}{\mathbb E}
\newcommand{\NN}{\mathbb N}

\renewcommand{\P}{{\mathsf P}}

\newcommand{\Fcal}{\mathcal F}

\newcommand{\Pcal}{\mathcal P}

\newcommand{\1}{\mathds{1}}

\newcommand{\seq}[4]{(#1)_{{#2}={#3}}^{#4}}

\newcommand{\infseqn}[1]{(#1)_{n \in \NN}}

\newcommand{\infseqm}[1]{\seq{#1}{m}{1}{\infty}}

\newcommand{\ci}{CI}

\newcommand{\cs}{confidence sequence}

\newcommand{\dd}{\mathrm{d}}
\newcommand{\Var}{\mathrm{Var}}

\newcommand{\oPas}{\widebar o_\P}%

\newcommand{\oPcalas}{\widebar o_\Pcal}%
\newcommand{\OPcalas}{\widebar O_\Pcal}%
\newcommand{\oPnullas}{\widebar o_{\Pcal_0}}
\newcommand{\OPnullas}{\widebar O_{\Pcal_0}}

\newcommand{\oP}{{\dot o_{\P}}}
\newcommand{\OP}{{\dot O_{\P}}}
\newcommand{\oPcal}{\dot o_\Pcal}
\newcommand{\OPcal}{\dot O_\Pcal}

\newcommand{\iid}{\text{i.i.d.}}

\newcommand{\indep}{\perp \!\!\! \perp}

\usepackage{setspace}

\newcommand{\Pnull}{{\Pcal_0}}
\newcommand{\Palt}{\Pcal_1}

\newcommand{\as}{{\mathrm{a.s.}}}
\newcommand{\eps}{\varepsilon}

\newcommand{\LP}{{L_2(\P)}}

\newcommand{\I}{\mathrm{I}}
\newcommand{\II}{\mathrm{II}}
\newcommand{\III}{\mathrm{III}}
\newcommand{\IV}{\mathrm{IV}}

\newcommand{\GCM}{\mathrm{GCM}}
\newcommand{\GCMSP}{\dot{\mathrm{GCM}}}
\newcommand{\GCMWS}{{\widebar{\mathrm{GCM}}}}

\newcommand{\brackm}{{(m)}}

\defcitealias{shah2020hardness}{S\&P20}

\newcommand{\infseqkm}[1]{\seq{#1}{k}{m}{\infty}}

\newcommand{\ubar}[1]{\underaccent{\bar}{#1}}

\newcommand{\rsrv}{\zeta}
\newcommand{\Psiplus}{\Psi_+}
\newcommand{\psiplus}{\psi_+}

\newcommand{\xy}{{x, y}}
\newcommand{\yx}{{y, x}}

\newcommand{\km}{{k,m}}

\newtheorem{assumptionMain}{Assumption}

\newtheorem{assumptionGCM}{Assumption}

\newtheorem{assumptionSeqGCM}{Assumption}

\newcommand{\highlightcell}[2]{%
  \setlength{\fboxrule}{2pt}%
  \fcolorbox{#1}{white}{#2}}

\renewcommand{\arraystretch}{1.7} %

\newcommand{\Pin}{{\P \in \Pcal}}
\newcommand{\Pnullin}{{\P \in \Pnull}}

\newcommand{\mto}{{m \to \infty}}
\newcommand{\xto}{{x \to \infty}}
\newcommand{\tth}{{\text{th}}}

\newcommand{\supkm}{{\sup_{k \geq m}}}

\newcommand{\pospart}[1]{\left [ #1 \right ]_+}

\newcommand{\abs}[1]{\left \lvert #1 \right \rvert}

\newcommand{\probspace}{(\Omega, \Fcal, \P)}

\newif\ifverbose %
\verbosefalse %

\usepackage{authblk}

\newcommand{\proofpreamble}{Proof }

\title{Distribution-uniform anytime-valid inference\\and the Robbins-Siegmund distributions}
\author{
  Ian Waudby-Smith$^\dagger$, Edward H. Kennedy$^\ddagger$, and Aaditya Ramdas$^\ddagger$\\
  $^\dagger$ University of California, Berkeley\\
  $^\ddagger$ Carnegie Mellon University\\
  \texttt{ianws@berkeley.edu \{edward,aramdas\}@stat.cmu.edu}
}
\date{}

\begin{document}

\maketitle
\setcounter{tocdepth}{2}
\makeatletter
\renewcommand\tableofcontents{%
  \@starttoc{toc}%
}

\makeatother

\begin{abstract}
  This paper develops a theory of distribution- and time-uniform asymptotics, culminating in the first large-sample anytime-valid inference procedures that are shown to be uniformly valid in a rich class of distributions.
  Historically, anytime-valid methods --- including confidence sequences, anytime $p$-values, and sequential hypothesis tests --- have been justified nonasymptotically. By contrast, large-sample inference procedures such as those based on the central limit theorem occupy an important part of statistical toolbox due to their simplicity, universality, and the weak assumptions they make. While recent work has derived asymptotic analogues of anytime-valid methods, they were not distribution-uniform (also called \emph{honest}), meaning that their type-I errors may not be uniformly upper-bounded by the desired level in the limit. 
  The theory and methods we outline resolve this tension, and they do so without imposing assumptions that are any stronger than the distribution-uniform fixed-$n$ (non-anytime-valid) counterparts or distribution-pointwise anytime-valid special cases.
  It is shown that certain ``Robbins-Siegmund'' probability distributions play roles in anytime-valid asymptotics analogous to those played by Gaussian distributions in standard asymptotics. 
  As an application, we derive the first anytime-valid test of conditional independence without the Model-X assumption.
\end{abstract}

\tableofcontents

\section{Introduction}
Some of the simplest and most ubiquitous statistical methods are asymptotic ones that rely on large-sample theory such as the central limit theorem (CLT). However, there is a distinction between asymptotics that are only valid for a single distribution $\P$ and those that are \emph{uniformly valid} over a large collection of distributions $\Pcal$. To elaborate, consider the classical CLT which states that for independent and identically distributed random variables $X_1, \dots, X_n \sim \P$ with mean $\mu_\P$ and finite variance $\sigma^2_\P < \infty$, their scaled partial sums $\widehat Z_n := \sum_{i=1}^n (X_i - \mu_\P) / (\sigma_\P \sqrt{n})$ are asymptotically standard Gaussian, meaning for any real $x$, we have $\P(\widehat Z_n \leq x) \to \Phi(x)$ where $\Phi$ is the cumulative distribution function (CDF) of a standard Gaussian. 
 Then, the difference between $\P$-pointwise and $\Pcal$-uniform convergence in distribution is:
\begin{equation}\label{eq:p-pointwise-vs-p-uniform-clt}
  \forall x,\ \underbrace{\sup_{\P \in \Pcal}\lim_{n \to \infty }\left \lvert \P (  \widehat Z_n \leq x ) - \Phi(x) \right \rvert = 0}_{\text{$\P$-pointwise convergence in distribution}}\quad\text{versus}\quad \forall x,\ \underbrace{\lim_{n \to \infty }\sup_{\P \in \Pcal}\left \lvert \P  (  \widehat Z_n \leq x  ) - \Phi(x) \right \rvert = 0}_{\text{$\Pcal$-uniform convergence in distribution}}.
\end{equation}
If $\Pcal$ denotes all distributions with finite mean and variance, then $\P$-pointwise convergence does hold for every $\P\in\Pcal$, but $\Pcal$-uniform convergence does not. However, if $\Pcal$ is restricted appropriately (in explicit ways we will return to)
then uniform convergence could hold.

If $\Pcal$-uniform convergence does not hold, then no matter how large $n$ is, $|{\P'}(\widehat Z_n \leq x) - \Phi(x)|$ can be far from zero for \emph{some} $\P' \in \Pcal$. An informal interpretation of the previous sentence is that asymptotic approximations may ``kick in'' arbitrarily late.
By contrast, \emph{distribution-uniformity} rules out the aforementioned scenario.

The initial study of $\Pcal$-uniformity is often attributed to \citet{li1989honest} and many papers have emphasized its importance in recent years; see \citet{kasy2018uniformity,rinaldo2019bootstrapping,tibshirani2018uniform,kuchibhotla2022median}, and \citet{lundborg2022conditional}. We highlight in particular the influential paper of~\citet{shah2020hardness}, which pointed out the critical role of distribution-uniform asymptotics in conditional independence testing, a problem that we will focus on in \cref{section:SeqCIT} of this paper.

\begin{remark}[On the terms ``honesty'' and ``uniformity'']\label{remark:honesty}
  The literature sometimes refers to distribution-uniformity as ``honesty'' \citep{li1989honest,kuchibhotla2022median} or simply ``uniformity'' \citep{kasy2018uniformity,rinaldo2019bootstrapping,tibshirani2018uniform,shah2020hardness,lundborg2022conditional}. We opt for the phrase ``distribution-uniform'' --- or ``$\Pcal$-uniform'' when we want to specify that uniformity is with respect to $\Pcal$ --- since there are many other notions of uniformity throughout probability and statistics, including time-uniformity and quantile-uniformity, both of which will become relevant throughout this paper. We do not use the term ``honesty'' as it has also been used to refer to other properties of estimators in statistical inference \citep{wager2018estimation,athey2019generalized}.
\end{remark}

Simultaneously, there is a parallel literature on \emph{anytime-valid} inference where the goal is to derive confidence sequences (CSs), anytime $p$-values, and sequential hypothesis tests (to be defined more formally later) which should be viewed as time-uniform analogues of confidence intervals (CIs), $p$-values, and hypothesis tests respectively that remain valid  when the sample size is a data-dependent stopping time. The overwhelming majority of the literature on anytime-valid inference has taken a nonasymptotic approach to inference so that the type-I errors and coverage probabilities hold in finite samples; see the early work of Wald, Robbins, and colleagues \citep{wald1945sequential,darling1967confidence,robbins1970statistical,lai1976confidence},
as well as the review paper of \citet{ramdas2022game} which gives a broad overview of this literature. However, nonasymptotic approaches generally require relatively strong assumptions on the random variables such as lying in a parametric family, or having \emph{a priori known} bounds on their support or on their moments.
On the other hand, asymptotic procedures such as those based on the CLT only require the \emph{existence} of a finite polynomial moment. The theory and methods developed in this paper adopt the latter approach to statistical inference where type-I errors and coverage probabilities hold in a limiting sense.

To illustrate time-uniformity in the asymptotic regime, we follow~\citet{waudby2021time}. Let $\infseqn{X_n}$ be an \iid{} sequence of random variables on a space $(\Omega, \Fcal, \P)$ with mean $\mu_\P \in \RR$ and variance $\sigma_\P^2 > 0$. We contrast the ``fixed-time'' coverage guarantee of an asymptotic CI $\infseqn{\dot C_n}$ with that of a ``time-uniform'' asymptotic CS $\infseqm{\infseqkm{\widebar C_k^\brackm}}$ as
\begin{equation}\label{eq:cs-vs-ci}
  \underbrace{\limsup_{n \to \infty }\P \left (  \mu_\P \notin \dot C_n \right ) \leq \alpha}_{\text{(Asymptotic) fixed-$n$ CI}}\quad\text{versus}\quad \underbrace{\limsup_{m \to \infty }\P \left ( \exists k \geq m : \mu_\P \notin \widebar C_k^\brackm \right ) \leq \alpha}_{\text{(Asymptotic) anytime-valid CS}}.
\end{equation}
In \eqref{eq:cs-vs-ci}, the main difference lies in the fact that the right-hand side probability holds uniformly in $k \geq m$ for sufficiently large $m$.
From a practical perspective, the right-hand side permits a researcher to continuously monitor the outcome of an experiment, for example, updating their confidence sets as each new data point is collected as long as the starting sample size $m$ is sufficiently large. Importantly, these anytime-valid procedures allow for the experiment to stop \emph{as soon as} the researcher has sufficient evidence to reject some null hypothesis (e.g.~as soon as $0 \notin \widebar C_k^\brackm$ for a null effect of 0). We henceforth omit the ``asymptotic'' phrase when referring to asymptotic procedures such as those in \eqref{eq:cs-vs-ci} since we are solely interested in asymptotics in this paper.
We now summarize our main results informally.
\begin{tcolorbox}
  \begin{theorem}[Distribution-uniform anytime-valid inference (brief \& informal)]\label{theorem:main-result-P-n-alpha-uniform-inference}
    Let $\infseqn{X_n}$ be a sequence of \iid{} random variables defined on 
    $(\Omega, \Fcal, \P)_\Pin$. Suppose that for some $\delta > 0$,
    \begin{equation}\label{eq:uniform-assumptions}
      \sup_\Pin \EE_\P[|X_1 - \mu_\P|^{2+\delta}] < \infty\quad\text{and}\quad \inf_\Pin \Var_\P (X_1) > 0.
    \end{equation}
    Define the set $\widebar C_k^\brackm(\alpha) := \widehat \mu_k \pm \widehat \sigma_k \sqrt{[\Psi^{-1}(1-\alpha) + \log(k/m)] / k}$ for any $\alpha \in (0, 1)$ and any $k \geq m$ where $\widehat \mu_k$ and $\widehat \sigma_k$ are the sample mean and sample standard deviation, respectively, and where $\Psi$ is the distribution function of the Robbins-Siegmund distribution to be introduced in \cref{definition:robbins-siegmund-distribution}. Then, 
  \begin{align}
    \forall \alpha \in (0, 1) \quad \lim_{m \to \infty} \sup_{\P \in \Pcal} \P \left ( \exists k \geq m : \mu_\P \notin \widebar C_k^\brackm(\alpha)  \right ) = \alpha.\label{eq:main-result-cs-conv}
  \end{align}
  Furthermore, defining the random variable $\widebar p_k^\brackm := 1- \Psi \left ( k\widehat \mu_k^2 / \widehat \sigma_k^2 - \log (k / m) \right )$, we have
  \begin{equation}
    \forall \alpha \in (0, 1), \quad \lim_\mto \sup_\Pin \P \left ( \exists k \geq m : \widebar p_k^\brackm \leq \alpha \right ) = \alpha.
  \end{equation}
  \end{theorem}
\end{tcolorbox}

In several works on $\Pcal$-uniform fixed-$n$ inference --- including Theorem 6 in the influential work of \citet{shah2020hardness} --- the exact conditions of \eqref{eq:uniform-assumptions} are assumed to conclude that
\begin{equation}\label{eq:fixed-n-coverage}
  \lim_{n\to \infty} \sup_\Pin \P \left ( \mu_\P \notin \dot C_n(\alpha) \right ) = \alpha,
\end{equation}
for the confidence interval $\dot C_n(\alpha) := \widehat \mu_n \pm \widehat \sigma_n \Phi^{-1}(1-\alpha/2)/\sqrt{n}$ (and analogously for a fixed-$n$ $p$-value). As such, both \eqref{eq:main-result-cs-conv} and \eqref{eq:fixed-n-coverage} follow from the same moment assumptions in \eqref{eq:uniform-assumptions} despite the former being a strictly stronger guarantee.

Notice that the informal theorem above is an asymptotic statement not in terms of typical (lower) triangular arrays but in terms of the \emph{upper} triangular arrays $\infseqm{\infseqkm{\widebar C_k^\brackm}}$ and $\infseqm{\infseqkm{\widebar p_k^\brackm}}$. We expand on this setting more formally in \cref{remark:upper triangular arrays} and justify our phrasing ``upper triangular array'' therein.
Arriving at a statement like the above requires certain technical results that are not common in typical asymptotic statistical theory. Indeed, some of the requisite probabilistic tools are the focus of recent additions to the literature by a subset of the present paper's authors \citep{waudby2024distribution,waudby2025nonasymptotic} and others will be developed here in a bespoke manner.

\subsection{Paper Outline}
Below we outline how the paper will proceed.
\begin{itemize}
  \item We begin in \cref{section:preliminaries} by outlining several mathematical foundations and preliminaries that will be instrumental for understanding and proving the guarantees of our main statistical results. In particular, we provide a notion of distribution-uniform almost-sure consistency in \cref{section:consistency} and show how variances can be estimated as such under weak moment assumptions. We define the Robbins-Siegmund distributions in \cref{section:robbins-siegmund-distribution} and show how they can be formed from transforming continuous-time Wiener processes or by taking the limit of transformed Gaussian sums. In \cref{section:strong-gaussian-coupling}, we review recent results on strong Gaussian approximation and use them in \cref{section:time-quantile-distribution-uniform clt} to show that transformed (non-Gaussian) partial sums converge distribution-uniformly to the aforementioned Robbins-Siegmund distributions.

    \item In \cref{section:distribution-uniform-anytime-valid-inference}, we take the mathematical preliminaries outlined in \cref{section:preliminaries} to both motivate and provide definitions for distribution-uniform confidence sequences, anytime $p$-values, and sequential hypothesis tests. We then arrive at our main methodological contributions: distribution-uniform large-sample statistical procedures for means of random variables with uniformly bounded polynomial moments.

    \item \cref{section:SeqCIT} illustrates an application of the above methodology by deriving what seem to be the first sequential tests of conditional independence without Model-X assumptions. We provide an accompanying impossibility result which shows that sequential tests of conditional independence are impossible to derive without imposing structural assumptions, a fact that can be viewed as a time-uniform analogue of the hardness result due to \citet[\S 2]{shah2020hardness}. Our main algorithm takes the form of a sequential leave-one-out analogue of the Generalized Covariance Measure test due to \citet{shah2020hardness}.

\end{itemize}

\subsection{Notation}
Throughout, we will let $\Omega$ be a sample space, $\Fcal$ the Borel sigma-algebra, and $\Pcal$ a collection of probability measures so that $(\Omega, \Fcal, \P)$ is a probability space for each $\P \in \Pcal$. We will write $(\Omega, \Fcal, \Pcal)$
to refer to the collection of probability spaces $(\Omega, \Fcal, P)_\Pin$.

For any event $A \in \Fcal$, we use $\P(A)$ to denote the probability of that event. We use $\EE_\P(X)$ to denote the expectation of a random variable $X$ on $(\Omega, \Fcal)$ with respect to $\P \in \Pcal$ (provided it exists). Similarly, $\Var_\P(X)$ will be shorthand for $\EE_\P(X - \EE_\P(X))^2$.

\section{Mathematical foundations and preliminaries}\label{section:preliminaries}

In order to arrive at the main results to be found in \cref{section:distribution-uniform-anytime-valid-inference}, we require a number of mathematical preliminaries that are not commonly encountered in the statistical literature. In some of these cases, the preliminaries rely on some recent probabilistic results on strong laws and Gaussian approximation due to a subset of the authors. In other cases, the preliminaries seem to be new. In \cref{section:consistency}, we recall a notion of distribution-uniform almost sure convergence due to \citet{chung2008strong}, and generalize it slightly to encompass distribution-uniform strong consistency of estimators. Using this definition, we provide a result on strongly consistent variance estimation that will be used later on in the proofs of the main results. In \cref{section:robbins-siegmund-distribution}, we define the one- and two-sided Robbins-Siegmund distributions and derive important properties thereof. In \cref{section:strong-gaussian-coupling}, we present a recent definition and theorem due to \citet{waudby2025nonasymptotic} surrounding distribution-uniform strong Gaussian approximation.

\subsection{Distribution-uniform almost-sure convergence and consistency}\label{section:consistency}

The results of \cref{section:distribution-uniform-anytime-valid-inference,section:SeqCIT}  implicitly rely on time-uniform analogues of some distribution-uniform convergence lemmas. In particular, we now review an old notion of distribution-uniform \emph{almost sure} convergence and prove several elementary but crucial properties which will be used later.

Let us first recall the classical notion of convergence in $\P$-probability for a single $\P \in \Pcal$ and its natural extension to $\Pcal$-uniform convergence in probability. That is, a sequence of random variables $X_1, X_2, \dots$ is said to converge \emph{in probability} to 0 (or $X_n = \oP(1)$ for short) if
\begin{equation}\label{eq:pointwise-convergence-in-probability}
   \forall \eps > 0,\quad \sup_\Pin\lim_{n \to \infty}\P(|X_n| \geq \eps) = 0,
\end{equation}
and that this convergence holds uniformly in $\Pcal$ (or $X_n = \oPcal(1)$ for short) if
\begin{equation}\label{eq:uniform-convergence-in-probability}
  \forall \eps > 0,\quad \lim_{n \to \infty} \sup_\Pin \P(|X_n| \geq \eps) = 0.
\end{equation}
The extension of \eqref{eq:pointwise-convergence-in-probability} to \eqref{eq:uniform-convergence-in-probability} is natural, but at first glance, an analogous extension for almost-sure convergence is less obvious. Indeed, recall that a sequence of random variables $X_1, X_2, \dots$ is said to converge $\P$-almost surely to 0 for every $\P \in \Pcal$ if
\begin{equation}\label{eq:pointwise-convergence-almost-surely}
  \forall \P \in \Pcal,\ \P \left (\lim_{n \to \infty} |X_n| = 0\right ) = 1.
\end{equation}
It is not immediately obvious what the ``right'' notion of $\Pcal$-uniform almost-sure consistency ought to be since taking an infimum over $\Pin$ of the above probabilities does not change the statement of \eqref{eq:pointwise-convergence-almost-surely} whatsoever. Intuitively, it is not possible to simply swap limits and suprema in \eqref{eq:pointwise-convergence-almost-surely} as was done when \eqref{eq:pointwise-convergence-in-probability} was extended to \eqref{eq:uniform-convergence-in-probability}.
However, it \emph{is} possible to make such a leap when using an equivalent definition of almost-sure convergence. To elaborate, it is a well-known fact that for any $\P \in \Pcal$,
\begin{equation}\label{eq:equivalence-as-convergence}
  \P \left (\lim_{n \to \infty} |X_n| = 0\right ) = 1 \quad\text{if and only if}\quad\forall \eps > 0,\ \lim_{m \to \infty}\P \left ( \sup_{k \geq m}|X_k| \geq \eps \right ) = 0.%
\end{equation}
In words, almost sure convergence is time-uniform convergence in probability, 
and for this reason, instead of writing $X_n = o_\as(1)$ as a shorthand for $\P$-almost-sure convergence, we write $X_n = \oPas(1)$ with the overhead bar $\widebar o$ to emphasize time-uniformity and the subscript $o_\P$ to emphasize the distribution $\P$ that this convergence is with respect to. 
Now, a natural notion of $\Pcal$-uniform almost-sure convergence is one that places a supremum over $\Pin$ in the right-hand limit of \eqref{eq:equivalence-as-convergence}. In particular, \citet{chung2008strong} states that a sequence $\infseqn{X_n}$ converges to 0 almost surely and distribution-uniformly if
\begin{equation}
  \forall\eps > 0,\quad \lim_\mto \sup_\Pin \P \left ( \supkm |X_k| \geq \eps \right ) = 0.
\end{equation}

\paragraph{Distribution-uniform strongly consistent estimation}
In what follows, we slightly generalize the aforementioned notion of uniform strong convergence due to \citet{chung2008strong} to the case of estimators of estimands that may themselves depend on each distribution. There is little conceptual difference between the following and Chung's notion, but the technical difference will be important once it is applied to statistical inference later on. 
\begin{definition}[$\Pcal$-uniform strong consistency]
  \label{definition:almost-sure-convergence}
  We say that an estimator $\infseqn{\widehat \theta_n}$ is $\Pcal$-uniformly consistent for the real-valued parameter $(\theta(\P))_{\Pin}$ if
\begin{equation}\label{eq:almost-sure-convergence}
  \forall \eps > 0,\quad \lim_{m \to \infty}\sup_\Pin \P \left ( \sup_{k \geq m}|\widehat \theta_k - \theta(\P)| \geq \eps \right ) = 0.
\end{equation}
As a shorthand for the above, we write $\widehat \theta_n - \theta = \oPcalas(1)$. Furthermore, we write $\widehat \theta_n - \theta = \oPcalas(r_n)$ for a monotonically nonincreasing and positive sequence $\infseqn{r_n}$ if $r_n^{-1} (\widehat \theta_n - \theta) = \oPcalas(1)$.
\end{definition}

\cref{table:four-notions-of-convergence} summarizes the four notions of convergence $\oP(\cdot)$, $\oPas(\cdot)$, $\oPcal(\cdot)$, and $\oPcalas(\cdot)$ and the implications between them. 
\renewcommand{\arraystretch}{1.5}
\begin{table}[!htbp]
  \caption{Four notions of convergence with implications between them. 
  If a sequence of random variables converges with respect to one of the four cells below, it also does so with respect to the cell above and/or to the left of it. This section is concerned with the strongest of the four, found in the bottom right cell with the \textbf{bolded} frame: $\Pcal$-uniform almost-sure convergence.}
  \label{table:four-notions-of-convergence}
  \centering
\begin{tabular}{l|>{\centering\arraybackslash}c>{\centering\arraybackslash}c>{\centering\arraybackslash}c}
    & \textbf{$\P$-pointwise} & & \textbf{$\Pcal$-uniform} \\
    \hline
    \textbf{In probability} & $\oP(\cdot )$  & $\impliedby$ & $\oPcal(\cdot )$ \\
     &  $\Uparrow$ & & $\Uparrow$ \\
   \textbf{Almost surely} & $\oPas(\cdot )$ & $\impliedby$ & \highlightcell{black}{$\oPcalas(\cdot )$}  \\
  \end{tabular}
\end{table}

For some of the main results to come, we will require that population variances can be strongly consistently estimated in a multiplicative sense. For such a consistency result, we will make the following assumption.
\begin{assumptionMain}\label{assumption:upper bound on variance}
  Let $X$ be a random variable. There exists some $\delta > 0$ for which
  \begin{equation}
    \sup_\Pin \EE_\P \left [ |X - \EE_\P[X]|^{2 + \delta} \right ] < \infty \quad\text{and}\quad \inf_\Pin \Var_\P (X) > 0.
    \end{equation}
\end{assumptionMain}
The conditions of \cref{assumption:upper bound on variance} are precisely those that appear in \eqref{eq:uniform-assumptions} but we re-state them formally here as we will routinely make use of them in the sections to follow. With \cref{assumption:upper bound on variance} in mind, we state the following result.

\begin{proposition}[$\Pcal$-uniform almost-sure multiplicative consistency of the variance]\label{theorem:consistent-variance-estimation}
  Suppose that $\infseqn{X_n}$ is a sequence of \iid{} random variables satisfying \cref{assumption:upper bound on variance}. Then $\widehat \sigma_n^2 / \sigma^2 - 1 = \oPcalas(n^{-\beta})$
  for $\beta = 1-2 / (2 + \delta / 2)$,
  or more explicitly, for all $\eps > 0$, we have 
  \begin{equation}
    \lim_{m \to \infty} \sup_\Pin \P \left ( \sup_{k \geq m} k^{1-2/(2+\delta / 2)}\left \lvert \frac{ \widehat \sigma_k^2}{\sigma_\P^2} - 1 \right \rvert \geq \eps \right )=0.
  \end{equation}
\end{proposition}
The proof of \cref{theorem:consistent-variance-estimation} can be found in \cref{proof:consistent-variance-estimation} and it relies on a distribution-uniform strong convergence result of \citet[Corollary 1]{waudby2024distribution}.

\paragraph{A calculus of time- and distribution-uniform asymptotics}
  Following the relationship between $\oP$ and $\OP$ notation in the fixed-$n$ in-probability setting, we now provide an analogous definition of time- and $\Pcal$-uniform stochastic boundedness. To the best of our knowledge, this definition is new to the literature.
  \begin{definition}
    \label{definition:time-uniform-stochastic-boundedness}
    We say that a sequence of random variables $\infseqn{X_n}$ is \uline{time- and $\Pcal$-uniformly stochastically bounded} if
    for any $\delta > 0$, there exists some $C\equiv C(\delta) > 0$ and $M\equiv M(\delta) > 1$ so that for all $m \geq M$,
    \begin{equation}
      \sup_\Pin \P \left ( \exists k \geq m : |X_k| > C \right ) < \delta,
    \end{equation}
    and we write $X_n = \OPcalas(1)$ as a shorthand for the above. Similar to \cref{definition:almost-sure-convergence}, we write $X_n = \OPcalas(r_n)$ if $r_n^{-1} X_n = \OPcalas(1)$ for a monotonically nonincreasing and positive sequence $\infseqn{r_n}$.
  \end{definition}
  A related condition has also appeared in the context of conditional local independence testing as in \citet{christgau2023nonparametric}.
Note that we do not refer to \cref{definition:time-uniform-stochastic-boundedness} as $\Pcal$-uniform ``almost sure'' boundedness since even in the $\P$-pointwise case, almost-sure boundedness and time-uniform stochastic boundedness are not equivalent despite the relationship in \eqref{eq:equivalence-as-convergence} for almost-sure and time-uniform \emph{convergence}.
  There is a calculus of $\oPcalas(\cdot)$ and $\OPcalas(\cdot)$ analogous to that for $\oPcal(\cdot)$ and $\OPcal(\cdot)$. We lay this out formally in the following lemma.
\begin{proposition}[Calculus of $\OPcalas(\cdot)$ and $\oPcalas(\cdot)$]\label{proposition:oPcalas-calculus}
  Let $\infseqn{X_n}$ be a sequence of random variables. Let $\infseqn{a_n}$ and $\infseqn{b_n}$ be positive and monotonically nonincreasing sequences. Then we have the following implications.
  \begin{align}
    X_n = \oPcalas(a_n) \implies& X_n = \OPcalas(a_n) \label{eq:oPcalas-implies-OPcalas}\\
    X_n = \oPcalas(a_n)\OPcalas(b_n) \implies& X_n = \oPcalas(a_nb_n)\label{eq:product-is-oPcalas}\\
    X_n = \OPcalas(a_n)\OPcalas(b_n) \implies& X_n = \OPcalas(a_nb_n)\label{eq:product-is-OPcalas}\\
  X_n =  \oPcalas(a_n) + \OPcalas(a_n) \implies& X_n = \OPcalas(a_n)\label{eq:sum-is-OPcalas}\\
    X_n =  \oPcalas(a_n) + \oPcalas(b_n) \implies& X_n = \oPcalas(\max\{a_n, b_n\})\label{eq:sum-of-oPcalas} \\
    X_n = \OPcalas(a_n) + \OPcalas(b_n) \implies& X_n = \OPcalas(\max\{a_n , b_n\}) \label{eq:sum-of-OPcalas}\\
    X_n = \OPcalas(a_n) \text{ and } a_n /b_n \to 0 \implies& X_n = \oPcalas(b_n)\label{eq:slightly-faster-convergence}\\ 
    X_n = \oPcalas(1) \implies& (1 + X_n)^{-1} = 1 + \oPcalas(1).\label{eq:reciprocal of 1 plus oPcalas}
  \end{align}
\end{proposition}
The proof of \cref{proposition:oPcalas-calculus} is routine but we provide it for completeness in \cref{proof:oPcalas-calculus}.
The calculus outlined in \cref{proposition:oPcalas-calculus} will appear frequently throughout the proofs of our main results.
In the next section, we introduce the Robbins-Siegmund distributions.

\subsection{The Robbins-Siegmund distributions}\label{section:robbins-siegmund-distribution}

Central to the methodological results to come are two probability distributions for suprema of transformed Wiener processes which were alluded to in the informal theorem provided in the introduction. To the best of our knowledge, the quantiles of these distributions were first implicitly computed by \citet{robbins1970boundary} and as such we refer to them as the \emph{Robbins-Siegmund distributions}. In this section, we define them explicitly and show how the suprema of scaled Wiener processes have these distributions.

\begin{definition}[The Robbins-Siegmund distributions]
  \label{definition:robbins-siegmund-distribution}
  Let $\Phi$ and $\phi$ be the CDF and density of a standard Gaussian random variable, respectively. Suppose that a random variable $\rsrv$ supported on $[0, \infty)$ has a CDF given by
  \begin{equation}
    \P(\rsrv \leq x) = 1 - 2 \left [ 1 - \Phi(\sqrt{x}) + \sqrt{x} \phi(\sqrt{x}) \right ]; \quad x \geq 0.
  \end{equation}
  Then we say that $\rsrv$ follows the \underline{two-sided Robbins-Siegmund distribution} and we denote its CDF by $\Psi(x) = \P (\rsrv \leq x)$ for $x \geq 0$ and $\Psi(x) = 0$ for $x < 0$.
  Alternatively, if the CDF of $\rsrv$ satisfies
  \begin{equation}
    \P(\rsrv \leq x) = \Phi(\sqrt{x}) - \phi(\sqrt{x}) \left ( \sqrt{x} + \frac{\phi(\sqrt{x})}{\Phi(\sqrt{x})} \right ),
  \end{equation}
  then we say that $\rsrv$ follows the 
  \uline{one-sided Robbins-Siegmund distribution} and denote its CDF by $\Psiplus(x) = \P(\rsrv \leq x)$ for $x \geq 0$ and $\Psiplus(x) = 0$ for $x < 0$.
\end{definition}
We remark that $\Psiplus$ is c\`adl\`ag despite the discontinuity at $0$. %
The two-sided Robbins-Siegmund distribution has also been implicitly used in \citet{bibaut2022near} and \citet{waudby2021time} for the sake of $\P$-pointwise anytime-valid inference, but here we aim to elevate its (and its one-sided counterpart's) status to central objects whose fundamental nature have not been fully appreciated in the literature thus far. We do not know of the one-sided Robbins-Siegmund distribution having been used for inferential purposes in the literature.

\begin{proposition}\label{proposition:properties of psi}
  The probability density function of the two-sided Robbins-Siegmund distribution is
  \begin{equation}
    \psi(x) := \frac{\dd \Psi(x)}{\dd x} = \sqrt{x} \phi(\sqrt{x});\quad x \geq 0.
  \end{equation}
  The probability density function of a one-sided Robbins-Siegmund random variable is
  \begin{equation}
    \psiplus(x) := \frac{\dd \Psiplus(x)}{\dd x} = \frac{\phi(\sqrt{x})}{2 \sqrt{x} \Phi(\sqrt{x})^2} \left ( \sqrt{x} \Phi(\sqrt{x}) + \phi(\sqrt{x}) \right )^2 \quad x > 0.
  \end{equation}
   Also, $\psi$ is uniformly bounded on $[0, \infty)$ and for any $a > 0$, $\psiplus$ is uniformly bounded on $[a, \infty)$. Consequently, $\Psi$ and $\Psiplus$ are Lipschitz continuous on $[0,\infty)$ and $[a,\infty)$, respectively.
\end{proposition}
We prove \cref{proposition:properties of psi} in \cref{proof:properties of psi}.
    Next, we show that $\Psi$ and $\Psiplus$ can be constructed from certain transformations of continuous time Wiener processes and that they are limiting distributions of certain transformations of Gaussian partial sums. 

\begin{proposition}[The Robbins-Siegmund distributions from transformed Gaussian processes]\label{proposition:robbins-siegmund from Gaussians}
  Let $g^\star : \RR \to \RR^+$ be the function given by $ g^\star (x) = x^2 + 2 \log (\Phi(x))$.
  We have the following four facts.
  \begin{enumerate}[label = (\roman*)]
  \item Let $(W(t))_{t\geq0}$ be a Wiener process. Then
    \begin{equation}
      \forall x \geq 0,\quad \P \left ( \sup_{t \geq 1} \left \{ \frac{W(t)^2}{t} - \log (t) \right \} \leq x \right ) = \Psi(x).
    \end{equation}
    \item 
    Furthermore, it holds that
    \begin{equation}
      \forall x > 0,\quad \P \left ( \sup_{t \geq 1} \left \{ g^\star \left ( \frac{W(t)}{\sqrt{t}} \right ) - \log(t) \right \} \leq x \right ) = \Psiplus(x).
    \end{equation}
  \item Let $G_k := \sum_{i=1}^k Y_i$ be a sum of $k$ \iid{} Gaussian random variables $Y_1, \dots, Y_k$ with mean zero and variance $\sigma^2$. Then,
    \begin{equation}
      \lim_\mto \sup_{x \geq 0} \left \lvert \P \left ( \supkm \left \{ \frac{G_k^2}{k \sigma^2} - \log \left ( \frac{k}{m} \right ) \right \} \leq x \right ) - \Psi(x) \right \rvert = 0.
    \end{equation}
  \item Furthermore,
  \begin{equation}
    \lim_\mto \sup_{x \geq 0} \left \lvert \P \left (   \supkm \left \{  g^\star \left ( \frac{G_k}{\sqrt{k\sigma^2} } \right ) - \log \left ( \frac{k}{m} \right ) \right \} \leq x  \right ) - \Psiplus(x) \right \rvert = 0.
  \end{equation}
  \end{enumerate}
\end{proposition}
A proof is provided in \cref{proof:robbins-siegmund from Gaussians} and relies on extending some results of \citet{robbins1970boundary} to the quantile-uniform setting.

\subsection{Distribution-uniform strong Gaussian approximation}\label{section:strong-gaussian-coupling}
The results alluded to in the introduction will be derived by showing that certain suprema of transformations of (typically non-Gaussian) data converge in distribution to the Robbins-Siegmund distributions in a similar spirit to \cref{proposition:robbins-siegmund from Gaussians}. At first glance, however, one might expect classical limit theorems (such as the CLT) to be poorly equipped for such pursuits as they are focused on providing statements about the large-sample behavior of statistics for a single sample size $n$. By contrast, the informal theorem presented in the introduction demands that such approximations hold uniformly in $k \geq m$ for large $m$. One set of tools that achieves time-uniform Gaussian approximations are so-called ``strong Gaussian approximations'' which can be viewed as analogues of CLTs that are written in terms of almost-sure couplings \citep{strassen1964invariance,strassen1967and,komlos1975approximation,komlos1976approximation}. Particularly relevant to the present paper is the recent work of \citet{waudby2025nonasymptotic} which provides \emph{distribution-uniform} versions of these types of couplings. Since we will rely on one of their results throughout the proofs of our main theorems, we present a simplification of it formally here.
\begin{theorem}[A simplification of Theorem 4 from \citep{waudby2025nonasymptotic}]\label{theorem:strong-gaussian-approx}
  Let $\infseqn{X_n}$ be \iid{} mean-zero random variables on the collection of probability spaces $(\Omega,\Fcal, \Pcal)$ and assume without loss of generality that for each $\Pin$, $\probspace$ is sufficiently rich as to describe Gaussian laws. Under \cref{assumption:upper bound on variance}, there exist mean-zero independent Gaussians $\infseqn{Y_n}$ on the same space with variance $\Var_{\P}(Y_1) = \Var_{\P}(X_1)$ so that
  \begin{equation}
    \forall \eps > 0, \quad \lim_\mto \sup_{\Pin}\P \left ( \supkm \left \lvert \frac{\sum_{i=1}^k ( X_i - Y_i) }{k^{1/q}} \right \rvert \geq \eps \right ) = 0.
  \end{equation}
\end{theorem}
Informally, the above result should be viewed as saying that under certain moment assumptions, the sample average $\frac{1}{n}\sum_{i=1}^n X_i$ is well approximated by an implicit sample average of \iid{} Gaussians $\frac{1}{n} \sum_{i=1}^n Y_i$ at a rate of $n^{1/q - 1}$ and uniformly in $\Pcal$. Using the notation of \cref{section:consistency}, we can alternatively write that under \cref{assumption:upper bound on variance}, there exists 
an implicit sum of \iid{} Gaussians $\infseqn{Y_n}$ so that $\sum_{i=1}^n X_i - \sum_{i=1}^n Y_i = \oPcalas(n^{1/q})$. 

\begin{remark}[On the phrase ``without loss of generality'' in strong approximations]\label{remark:strongapprox wlog}
  Strictly speaking, the result in \citet[Theorem 4]{waudby2025nonasymptotic} requires one to construct a new probability space $(\widetilde \Omega, \widetilde \Fcal, \widetilde \P)$ for each $\Pin$ since the initial space may not have been sufficiently rich to describe Gaussian laws. It is for this reason that we assumed without loss of generality that $(\Omega, \Fcal, \P)$ can do so. This is a common preamble in results on strong Gaussian approximation \citep{strassen1967and}. However, this technical subtlety will not matter for the main results to come since they will be written purely in terms of probabilities and events with respect to the original space.
\end{remark}

\subsection{Time-, quantile-, and distribution-uniform central limit theory}\label{section:time-quantile-distribution-uniform clt}

Recall that in the fixed-$n$ setting, the ($\Pcal$-uniform) CLT can often be found stated for a single quantile, meaning that the CDF $\P (S_n / (\sigma\sqrt{n}) \leq x)$ of $\sqrt{n}$-scaled partial sums $S_n := \sum_{i=1}^n [ X_i - \EE_\P(X) ]$ ($\Pcal$-uniformly) converges to that of a standard Gaussian:
\begin{equation}\label{eq:classical-pointwise-clt}
  \forall x \in \RR,\ \lim_{n \to \infty} \sup_\Pin \left \lvert  \P\left (S_n/(\sigma \sqrt{n}) \leq x \right ) - \Phi(x)   \right \rvert = 0,
\end{equation}
with $\Pcal = \{\P\}$ for the distribution-pointwise case.
Under no additional assumptions, however, the above holds \emph{quantile}-uniformly \citep[Lemma 2.11]{van2000asymptotic}, meaning that
\begin{equation}\label{eq:quantile-uniform-clt}
  \lim_{n \to \infty} \sup_\Pin \sup_{x \in \RR} \left \lvert  \P \left (S_n/ (\sigma \sqrt{n}) \leq x \right) - \Phi(x)   \right \rvert = 0.
\end{equation}
Clearly, \eqref{eq:quantile-uniform-clt} implies \eqref{eq:classical-pointwise-clt}. Particularly relevant to this paper, key steps in the proofs of distribution-uniform \emph{fixed-$n$} tests and \ci{}s rely on quantile-uniformity. Even in the $\P$-pointwise case, however, there is no result showing that such quantile-uniformity exists for time-uniform \emph{boundaries} (and it is not clear \emph{a priori} in what sense such a statement should or could be formulated). The following theorem provides such a result, making use of the Robbins-Siegmund distributions introduced in the previous section.
\begin{theorem}\label{theorem:main-convergence-in-distribution}
  Let $\infseqn{X_n}$ be \iid{} random variables satisfying \cref{assumption:upper bound on variance} and put $S_n := \sum_{i=1}^n X_i$ for $n \in \NN$.
  Then
  \begin{equation}\label{eq:main-theorem-two-sided}
    \lim_\mto \sup_\Pin \sup_{x \geq 0} \left \lvert \P \left ( \supkm \left \{ \frac{S_k^2}{\widehat \sigma_k^2 k} - \log \left ( \frac{k}{m} \right ) \right \} \geq x \right ) - [ 1 - \Psi(x) ] \right \rvert = 0.
  \end{equation}
  Letting $g(x) = \left ( x^2 + 2 \log \Phi(x) \right ) \lor 0$ for $x \in \RR$ and $a_1 := \Psiplus^{-1}(1/2)$, we also have that
  \begin{equation}\label{eq:main-theorem-one-sided}
    \lim_\mto \sup_\Pin \sup_{x \geq a_1} \abs{ \P \left ( \supkm \left \{ g \left ( \frac{S_k \lor 0}{\widehat \sigma_k \sqrt{k}}\right ) - \log \left ( \frac{k}{m} \right ) \right \} \geq x \right ) - [1-\Psiplus(x)] } = 0.
  \end{equation}
\end{theorem}

The sense in which \cref{theorem:main-convergence-in-distribution} can be viewed as ``boundary-uniform'' results is as follows. It is easy to deduce that the expression in \eqref{eq:main-theorem-two-sided} is equivalent to
\begin{equation}
  \lim_\mto \sup_\Pin \sup_{x\geq 0} \left \lvert \P \left ( \exists k \geq m : |S_k| \geq \widehat \sigma_k \sqrt{k \left (x + \log \left ( \frac{k}{m} \right )\right )} \right ) - [1 - \Psi(x)]\right \rvert = 0
\end{equation}
and the expression \eqref{eq:main-theorem-one-sided} is equivalent to
\begin{equation}
  \lim_\mto \sup_\Pin \sup_{x\geq a_1} \left \lvert \P \left ( \exists k \geq m : S_k \lor 0 \geq \widehat \sigma_k \sqrt{k} g^{-1} \left (x + \log \left ( \frac{k}{m} \right ) \right ) \right ) - [1 - \Psi(x)]\right \rvert = 0
\end{equation}
Taking the suprema over $\Pin$ and $x \geq 0$ (or $x \geq a_1$) out of the limits and replacing $\widehat \sigma_k$ with $\sigma$ recovers some results of \citet{robbins1970boundary}, essentially showing that their results can be made distribution- and boundary-uniform.
Given its centrality to the main results to come, we provide a brief proof sketch for the two-sided case in \eqref{eq:main-theorem-two-sided} but detailed proofs for both \eqref{eq:main-theorem-two-sided} and \eqref{eq:main-theorem-one-sided} can be found in \cref{proof:main-convergence-in-distribution}.
\begin{proof}[\proofpreamble{}sketch of \cref{theorem:main-convergence-in-distribution}]
  We aim to show that the difference in \eqref{eq:main-theorem-two-sided} is smaller than an arbitrary $\eps > 0$. Take $\delta > 0$ to be a small constant. First, we use \cref{theorem:consistent-variance-estimation} to take $m \in \NN$ sufficiently large so that
  \begin{equation}
    \sup_\Pin \P \left ( \supkm \left \lvert \frac{\widehat \sigma_k^2}{\sigma_\P^2} - 1 \right \rvert \geq \delta \right ) \ll \eps.
  \end{equation}
  Second, we use \citet[Theorem 4]{waudby2025nonasymptotic} as discussed in \cref{section:strong-gaussian-coupling} to take $m \in \NN$ large enough so that for some $\beta > 0$,
  \begin{equation}
    \sup_\Pin \P \left ( \supkm  \abs{\frac{S_k - G_k}{k^{1/\beta}}} \geq \delta \right ) \ll \eps.
  \end{equation}
  for a sum $G_n = \sum_{i=1}^n Y_i$ of \iid{} Gaussian random variables $\infseqn{Y_n}$ with mean zero and variance $\Var_\P(Y_1) = \sigma_\P^2$ for each $\Pin$. Taking these two approximations together, we have that
  \begin{equation}
    \P \left ( \supkm \left \{ \frac{S_k^2}{\widehat \sigma_k^2 k} - \log \left ( \frac{k}{m} \right ) \right \} \geq x \right ) \approx \P \left ( \supkm \left \{ \frac{G_k^2}{\sigma_\P^2 k} - \log \left ( \frac{k}{m} \right ) \right \} \geq x \pm \delta \right ) \pm \eps.
  \end{equation}
  Taking $m$ large, we have that for all $x \in \RR$,
 \begin{equation}
   \P \left ( \supkm \left \{ \frac{G_k^2}{\sigma_\P^2 k} - \log \left ( \frac{k}{m} \right ) \right \} \geq x \pm \delta \right ) \approx 1-\Psiplus(x \pm \delta),
 \end{equation} 
 and since $\Psiplus$ is uniformly continuous (indeed, Lipschitz, by \cref{proposition:properties of psi}), we have that $\Psiplus(x\pm\delta) \approx \Psiplus(x)$, which completes the sketch of the proof.
\end{proof}
With access to \cref{theorem:main-convergence-in-distribution} and the mathematical preliminaries outlined in this section, we are ready to state and provide proofs for the main methodological results of this paper.

\section{Distribution-uniform anytime-valid inference}\label{section:distribution-uniform-anytime-valid-inference}
Let us begin this section by providing motivation for a particular definition of $\Pcal$-uniform anytime $p$-values, sequential tests, and confidence sequences. Recall that for $\alpha \in (0, 1)$, a classical fixed-$n$ $(1-\alpha)$-confidence interval $\dot C_n$ is said to be $\Pcal$-uniformly valid for the parameters $(\mu_\P)_\Pin$ if
\begin{equation}\label{eq:distribution-uniform-asymptotic-ci}
  \limsup_{n \to \infty} \sup_{\P \in \Pcal} \P ( \mu_\P \notin \dot C_n ) \leq \alpha.
\end{equation}
On the other hand, \citet[Definition 2.7]{waudby2021time} provide a definition of ($\P$-pointwise) time-uniform $(1-\alpha)$-coverage given as follows:
\begin{equation}\label{eq:time-uniform-asymptotic-cs}
  \forall \Pin,\ \limsup_{m \to \infty} \P(\exists k\geq m : \mu_\P \notin \widebar C_k^\brackm ) \leq \alpha.
\end{equation}
The above guarantee has also appeared in \citet{robbins1970boundary} and \citet{bibaut2022near}.
Juxtaposing \eqref{eq:distribution-uniform-asymptotic-ci} and \eqref{eq:time-uniform-asymptotic-cs}, we can intuit the right definition of distribution- \emph{and} time-uniform type-I error control by placing a supremum over $\Pcal$ inside the limit in \eqref{eq:time-uniform-asymptotic-cs}. We now lay this definition out formally alongside corresponding definitions for anytime $p$-values and sequential hypothesis tests as well as sharpness properties thereof.
  \begin{definition}[$\Pcal$-uniform sequential tests, anytime $p$-values, and confidence sequences]
    \label{definition:distribution-time-uniform-type-I-error-control}
    Let $\Pcal$ be a collection of distributions and let $\Pcal_0 \subseteq \Pcal$ be a null hypothesis.
    For $\alpha \in (0, 1)$, we say that $\infseqm{\infseqkm{\widebar \Gamma_k^\brackm}}$ is a \uline{$\Pnull$-uniform level-$\alpha$ sequential hypothesis test} if
    \begin{equation}
      \limsup_{m \to \infty}\sup_\Pnullin \P \left (\exists k \geq m : \widebar \Gamma_k^\brackm = 1 \right) \leq \alpha.
    \end{equation}
 We say that $\infseqm{\infseqkm{\widebar p_k^\brackm}}$ is a \uline{$\Pnull$-uniform anytime $p$-value} if
\begin{equation}\label{eq:p-uniform-p-value}
      \forall \alpha \in (0, 1),\quad\limsup_{m \to \infty} \sup_\Pnullin \P \left (\exists k \geq m : \widebar p_k^\brackm \leq \alpha \right ) \leq \alpha.
\end{equation}
    For $\alpha \in (0, 1)$, we say that $\infseqm{\infseqkm{\widebar C_k^\brackm}}$ is a \uline{$\Pcal$-uniform $(1-\alpha)$-confidence sequence} for $(\theta_\P)_{\Pin}$ if
    \begin{equation}
      \limsup_{m \to \infty} \sup_{\P \in \Pcal} \P \left ( \exists k \geq m : \theta_\P \notin \widebar C_k^\brackm \right ) \leq \alpha.
    \end{equation}
    Finally, we additionally say that the above objects are \uline{sharp} if the limit suprema are limits and the inequalities $(\leq \alpha)$ are equalities $(= \alpha)$.
  \end{definition}
  \begin{remark}[Asymptotics of upper triangular arrays]\label{remark:upper triangular arrays}
    While lower triangular arrays are commonly encountered in asymptotic theory, the definitions above take the form of \emph{upper} triangular arrays. We use the aforementioned name because, for example, an anytime $p$-value $\infseqm{\infseqkm{\widebar p_k^\brackm}}$ can be visually depicted as
\begin{alignat}{3}
  & \widebar p_{2}^{(1)},\ & \widebar p_{3}^{(1)},\ &\widebar p_{4}^{(1)}, &\dots \\
  & & \widebar p_{3}^{(2)},\ &\widebar p_{4}^{(2)}, &\dots \\
  && &\widebar p_{4}^{(3)},\ &\dots \\
  &&&& \ddots
\end{alignat}
and so on. Many of the technical results in this paper can be viewed as new tools for analyzing random upper triangular arrays.
  \end{remark}
As one may expect, any $\Pcal$-uniform anytime-valid test, $p$-value, or \cs{} satisfying  \cref{definition:distribution-time-uniform-type-I-error-control} is also $\Pcal$-uniform for a fixed sample size $n$ in the sense of \eqref{eq:distribution-uniform-asymptotic-ci} as well as $\P$-pointwise anytime-valid for any $\P \in \Pcal$ in the sense of \eqref{eq:time-uniform-asymptotic-cs}.
With \cref{definition:distribution-time-uniform-type-I-error-control} in mind, we will now derive distribution-uniform anytime hypothesis tests, $p$-values, and confidence sequences for the mean of independent and identically distributed random variables.

\subsection{One- and two-sided inference for means of \iid{} random variables}\label{section:uniform-anytime-valid-tests-and-confidence-sequences-for-the-mean}
With the requisite definitions and mathematical preliminaries provided in the previous sections, we now use the Robbins-Siegmund distributions to derive one- and two-sided distribution-uniform inferential procedures for means. Let us begin with the two-sided case.

Let $\infseqn{X_n}$ be random variables, $\widehat \mu_n := \frac{1}{n} \sum_{i=1}^n X_i$ their sample mean, and $\widehat \sigma_n^2 := \frac{1}{n} \sum_{i=1}^n (X_i - \widehat \mu_n)^2$ their sample variance. Let $\Psi$ be the two-sided Robbins-Siegmund distribution (\cref{definition:robbins-siegmund-distribution}). Define the upper triangular array $\infseqkm{\widebar p_k^\brackm}$ given by 
  \begin{equation}
    \widebar p_k^\brackm := 1-\Psi \left ( k \widehat \mu_k^2 / \widehat \sigma_k^2 - \log(k / m) \right )
  \end{equation}
  and the upper triangular array of intervals $\infseqkm{\widebar C_k^\brackm(\alpha)}$  given by
  \begin{equation}
    \widebar C_k^\brackm(\alpha) := \widehat \mu_k \pm \widehat \sigma_k \sqrt{\frac{\Psi^{-1}(1-\alpha) + \log(k/m)}{k}}.
  \end{equation}
  The following result gives conditions under which $\infseqkm{\widebar p_k^\brackm}$ is a $\Pcal_0$-uniform anytime $p$-value for the null hypothesis that the mean is zero as well as conditions under which $\infseqkm{\widebar C_k^\brackm(\alpha)}$ is a $\Pcal$-uniform $(1-\alpha)$-\cs{} for the mean in the sense of \cref{definition:distribution-time-uniform-type-I-error-control}.
\begin{theorem}[Two-sided distribution-uniform anytime-valid inference]\label{theorem:two-sided inference}
  Let $\infseqn{X_n}$ be random variables satisfying \cref{assumption:upper bound on variance}.
  Let $\Pcal_0 := \{ \Pin : \EE_\P[X_1] = 0 \}$ be the subcollection of distributions for which $X_1$ has mean zero. Then $\infseqkm{\widebar p_k^\brackm}$ is a sharp $\Pcal_0$-uniform anytime $p$-value:
  \begin{equation}
    \lim_{m \to \infty} \sup_{\P \in \Pcal_0} \P \left ( \exists k \geq m : \widebar p_k^\brackm \leq \alpha \right ) = \alpha,
  \end{equation}
  and $\infseqkm{\widebar C_k^\brackm(\alpha)}$ is a sharp $\Pcal$-uniform \cs{} for the mean:
  \begin{equation}
    \lim_{m \to \infty} \sup_{\P \in \Pcal} \P \left ( \exists k \geq m : \EE_\P[X_1] \notin \widebar C_k^\brackm(\alpha) \right ) = \alpha.
  \end{equation}
\end{theorem}
A sharp $\Pcal_0$-uniform sequential test can be obtained by taking $\widebar \Gamma_k^\brackm := \1 \{ \widebar p_k^\brackm \leq \alpha \}$.
\begin{figure}[!htbp]
  \centering
  \includegraphics[width=0.49\textwidth]{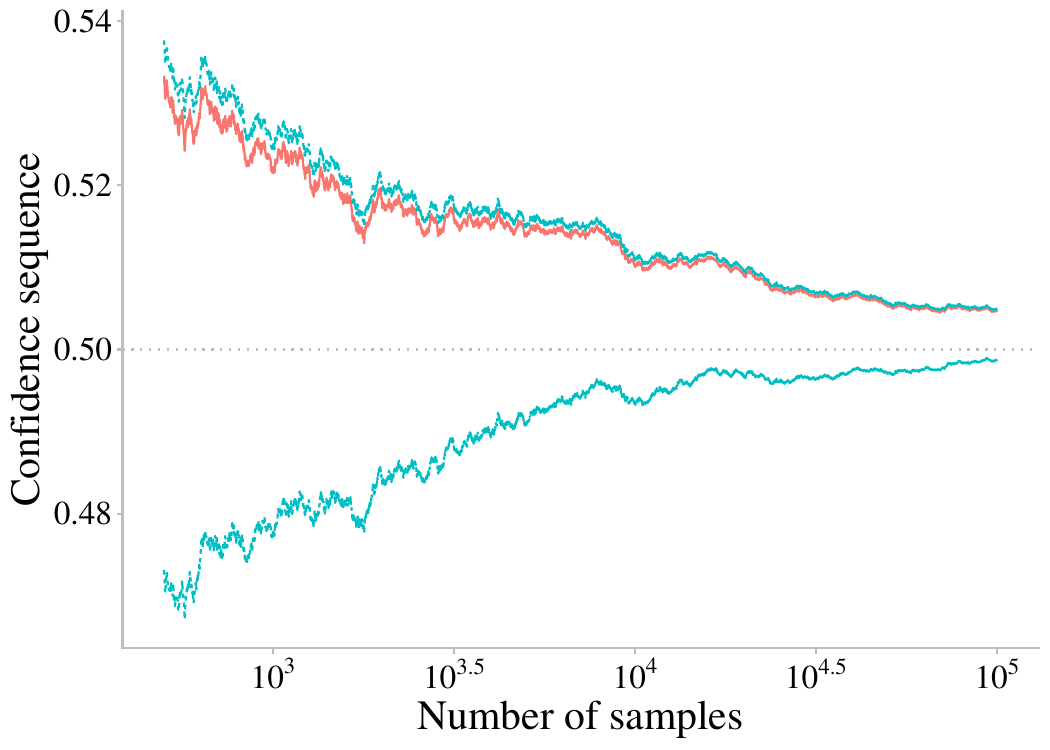}
  \includegraphics[width=0.49\textwidth]{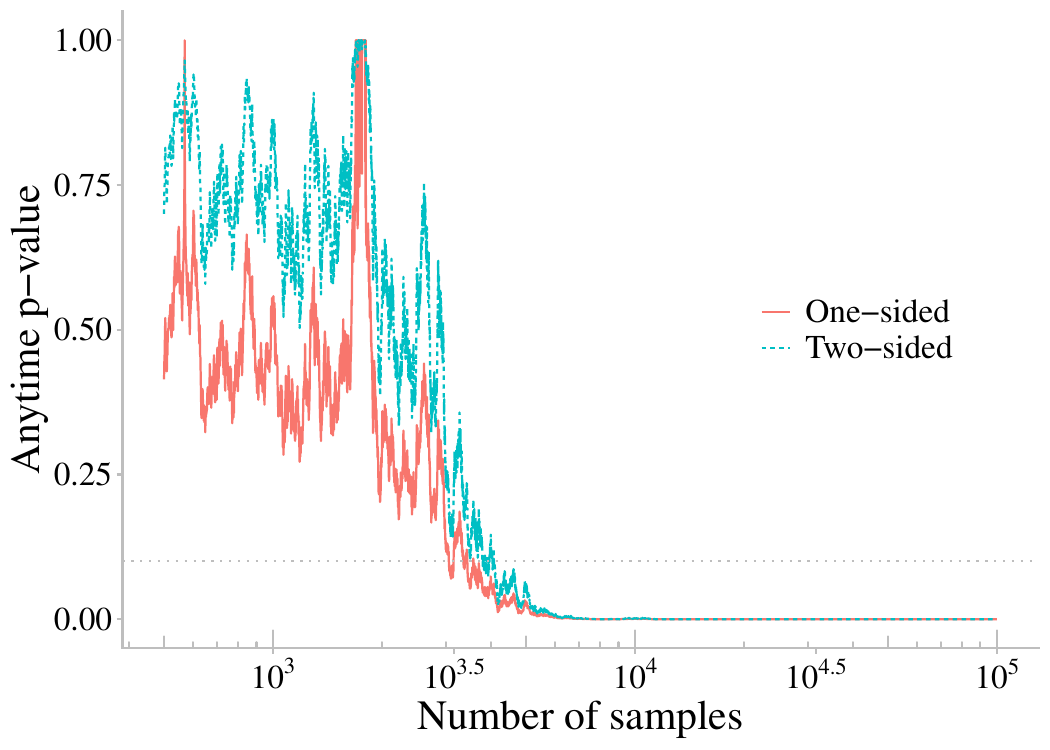}
  \caption{The left-hand side plot contains one- and two-sided confidence sequences for the mean of \iid{} Uniform$[0,1]$ random variables. The right-hand side plot contains one- and two-sided anytime $p$-values for the null hypothesis of $\EE_\P[X_1] = 1/2$ under the alternative $\EE_\P[X_1] = 0.51$.}
  \label{fig:two vs one sided}
\end{figure}
With access to the results in \cref{section:preliminaries}, the proof of \cref{theorem:two-sided inference} is short and can be found in \cref{proof:two-sided inference}.
Let us now move to the one-sided case. 
\begin{theorem}[One-sided distribution-uniform anytime-valid inference]\label{theorem:one-sided inference}
  Let $\infseqn{X_n}$ be random variables satisfying \cref{assumption:upper bound on variance}. Define $\Pcal_0^\leq$ as the subset of distributions with a non-positive mean $\Pcal_0^\leq := \left \{ \Pin : \EE_\P[X_1] \leq 0 \right \}$.
  Then the upper triangular array $\infseqm{\infseqkm{\widebar p_{k}^\brackm}}$ given by
  \begin{equation}
    \widebar p_k^\brackm := 1 - \Psiplus\left (g \left (\sqrt{k(\widehat \mu_k\lor 0)}/ \widehat \sigma_k\right ) - \log (k / m )\right )
  \end{equation}
  is a sharp $\Pcal_0^\leq$-uniform anytime $p$-value:
  \begin{equation}
    \forall \alpha \in (0, 1/2],\quad \lim_\mto \sup_{\P \in \Pcal_0^\leq} \P \left ( \exists k \geq m : \widebar p_k^\brackm \leq \alpha \right ) = \alpha.
  \end{equation}
\end{theorem}
The proof in \cref{proof:one-sided inference} is more involved but follows similarly to that of \cref{theorem:two-sided inference}.
See \cref{fig:two vs one sided} for an illustration of how the downstream confidence sequences and anytime $p$-values compare when derived from two- versus one-sided boundaries as in \cref{theorem:two-sided inference,theorem:one-sided inference}. Let us now use \cref{theorem:two-sided inference} to arrive at Model-X-free sequential tests of conditional independence.

\section{Illustration: Sequential conditional independence testing}\label{section:SeqCIT}

In this section, we aim to derive anytime-valid tests for the null hypothesis, $X \indep Y \mid Z$ given $\RR \times \RR \times \RR^d$-valued triplets $\infseqn{X_n, Y_n, Z_n}$ on probability spaces $(\Omega, \Fcal, \Pcal)$. Several works on conditional independence testing operate under the so-called ``Model-X'' assumption where the conditional distribution of $X \mid Z$ is known exactly \citep{candes2018panning}. We avoid the Model-X assumption in this paper. It is well-known that testing for conditional independence is much simpler under Model-X, and indeed the recent works of \citet{duan2022interactive}, \citet{shaer2023model}, and \citet{grunwald2023anytime} derive powerful (nonasymptotic) anytime-valid tests in that paradigm. Borrowing a quote from the recent work of \citet{grunwald2023anytime}, the authors write \emph{``it is an open
question to us how to construct general sequential tests of conditional independence without the
[Model-X] assumption''}. This section gives an answer to this question, deriving tests that draw inspiration from the batch tests found in \citet{shah2020hardness} --- a pair of authors we will henceforth refer to as \citetalias{shah2020hardness}. Before giving a brief refresher on batch conditional independence testing and the main results of \citetalias{shah2020hardness}, let us review some basic concepts in weak regression consistency since nuisance function estimation will form key conditions for our results.

\subsection{Prelude: weak regression consistency}

An important part of conditional independence testing (in both batch and sequential settings as we will see) is the ability to consistently estimate certain regression functions. Recall that the squared $\LP$ risk of a regression estimator $\widehat f_n : \RR^d \to \RR$ for a function $f : \RR^d \to \RR$ is given by
\begin{equation}
  \| \widehat f_n - f \|_\LP^2 := \int_{z \in \RR^d} \left ( \widehat f_n(z) - f(z) \right )^2 \dd \P(z) .
\end{equation}
Importantly, if sample splitting is used to construct $\widehat f_n$, the norm $\|\cdot \|_\LP$ is to be interpreted as conditional on that ``training'' data. Recall from \citet[Definition 1.1]{gyorfi2002distribution} that a regression estimator $\widehat f_n : \RR^d \to \RR$ is $\P$-weakly consistent for a function $f : \RR^d \to \RR$ in $\LP$ at a rate of $r_n$ if its expected squared $\LP$ risk vanishes at that rate, meaning
\begin{equation}\label{eq:weak-consistency-pointwise}
  \EE_\P \| \widehat f_n - f \|_\LP^2 = o(r_n),
\end{equation}
and hence we will say that $\widehat f_n$ is $\Pcal$-weakly consistent at the rate $r_n$ if the above convergence occurs uniformly in the class of distributions $\Pcal$:
\begin{equation}\label{eq:weak-consistency-uniform}
  \sup_\Pin \EE_\P \|\widehat f_n - f\|_\LP^2 = o(r_n).
\end{equation}
We may omit $\LP$ from  $\| \cdot \|_\LP$ in \eqref{eq:weak-consistency-pointwise} and write $\| \cdot \|$ when the norm is clear from context.

\subsection{A brief refresher on batch conditional independence testing}\label{section:cit-refresher}

Given $\RR \times \RR \times \RR^d$-valued triplets $(X_i, Y_i, Z_i)_{i=1}^n$ from some distribution in a class $\Pcal$, the problem of conditional independence testing is concerned with the null
\begin{equation}\label{eq:cit}
  H_0: X \indep Y \mid Z \quad\text{versus the alternative}\quad H_1: X\ \cancel{\indep}\ Y \mid Z.
\end{equation}
As alluded to before, without the Model-X assumption, powerful tests for the conditional independence null $H_0$ in \eqref{eq:cit} are \emph{impossible} to derive (even in the batch and asymptotic settings) unless additional distributional or structural assumptions are imposed \citepalias[\S 2]{shah2020hardness}. Indeed, \citetalias{shah2020hardness} show that even in the bounded setting where the triplets $(X,Y,Z)$ are supported on $[0,1]^3$, any test with distribution-uniform type-I error control under $H_0$ is powerless against \emph{any} alternative in $H_1$. Formally, if $\Pcal^\star$ is the set of distributions with support $[0, 1]^3$ and $\Pnull^\star \subset \Pcal^\star$ is the subset of distributions satisfying $H_0$ (and hence $\Palt^\star:= \Pcal^\star \setminus \Pnull^\star$ satisfies $H_1$), then
\newcommand{\sphardness}{\underbrace{\sup_{\P \in \Palt^\star} \limsup_{n \to \infty} \P \left ( \dot \Gamma_n = 1 \right )}_{\text{Best-case $\Palt^\star$-pointwise power}} \leq \underbrace{\limsup_{n \to \infty} \sup_{\P \in \Pcal_0^\star} \P \left ( \dot \Gamma_n = 1 \right ).}_{\text{Worst-case $\Pcal_0^\star$-uniform type-I error}}}
\begin{equation}\label{eq:best-case-power-worst-case-type-i-err}
  \sphardness
\end{equation}
As a consequence of \eqref{eq:best-case-power-worst-case-type-i-err}, one cannot derive a more powerful test than the trivial one that ignores all of the data $\seq{X_i,Y_i,Z_i}{i}{1}{n}$ and randomly outputs 1 with probability $\alpha$.

Despite the rather pessimistic result in \eqref{eq:best-case-power-worst-case-type-i-err}, \citetalias{shah2020hardness} derive the Generalized Covariance Measure (GCM) test which manages to achieve nontrivial power while still uniformly controlling the type-I error. The caveat there is that they are controlling the type-I error in a restricted (but nevertheless rich and nonparametric) class of nulls $\Pnull \subseteq \Pcal_0^\star$, and the restriction they impose is that certain nuisance functions are sufficiently estimable, a requirement commonly appearing in other literatures including semiparametric functional estimation \citep{kennedy2022semiparametric,balakrishnan2023fundamental}. Let us now review the key aspects of their test. \citetalias{shah2020hardness} introduce the estimated residuals $R_{i,n}$ for each $i \in [n]$:
\begin{equation}\label{eq:shah-peters-residuals}
  R_{i,n}:=\left \{ X_i - \widehat \mu_n^x(Z_i) \right \} \left \{ Y_i - \widehat \mu_n^y(Z_i) \right \}
\end{equation}
where $\widehat \mu_n^x(z)$ and $\widehat \mu_n^y(z)$ are estimates of the regression functions $\mu^x(z) := \EE_\P(X \mid Z=z)$ and $\mu^y(z) := \EE_\P(Y \mid Z=z)$. As their name suggests, the estimated residuals will serve as estimates for the true residuals $\xi_i := \xi_i^x\xi_i^y$ where
\begin{equation}
  \xi_i^x := \{ X_i - \mu^x(Z_i) \} \quad\text{and}\quad \xi_i^y := \{Y_i - \mu^y(Z_i) \}
\end{equation}
for each $i \in [n]$.
For the remainder of the discussion on batch conditional independence testing, we will assume that $\widehat \mu_n^x(Z_i)$ and $\widehat \mu_n^y(Z_i)$ are constructed from an independent sample (e.g.~through sample-splitting or cross-fitting, in which case we may assume access to $2n$ triplets of $(X, Y, Z)$) for mathematical simplicity, but \citetalias{shah2020hardness} do not always suggest doing so. However, we will not dwell on arguments for or against sample splitting here. From the residuals in \eqref{eq:shah-peters-residuals}, they construct the test statistic $\GCMSP_n$ taking the form
\begin{equation}
  \GCMSP_n := \frac{1}{n \widehat \sigma_n}\sum_{i=1}^n R_{i,n}
\end{equation}
where $\widehat \sigma_n^2 := \frac{1}{n} \sum_{i=1}^n R_{i,n}^2 - \left ( \frac{1}{n}\sum_{i=1}^n R_{i,n} \right )^2$ and they show that if the regression functions $(\mu^y, \mu^x)$ are estimated sufficiently fast (and under some other mild regularity conditions) then $\sqrt{n} \GCMSP_n$ has a standard Gaussian limit, enabling asymptotic (fixed-$n$) inference. We formally recall a minor simplification of their main result here. Consider the following three assumptions for a class of distributions $\Pnull$.

\begin{assumptionGCM}[Product regression error decay]
  \label{assumption:GCM-regression-product-errors}
The weak convergence rate of the product of expected residual norms is faster than $n^{-1/2}$, i.e.
  \begin{equation}
    \sup_\Pnullin \sqrt{\EE_\P\| \mu^x - \widehat \mu_n^x \|^2} \cdot \sqrt{\EE_\P \| \mu^y - \widehat \mu_n^y \|^2} = o(n^{-1/2}).
  \end{equation}
\end{assumptionGCM}

\begin{assumptionGCM}[$\Pnull$-uniform regularity of regression errors]
  \label{assumption:GCM-regression-estimators-do-not-diverge}
  Suppose that the variances of $\{ \widehat \mu_n^x(Z) - \mu^x(Z) \} \xi^y$ and $\{ \widehat \mu_n^y(Z) - \mu^y(Z) \} \xi^x$ are $\Pnull$-uniformly vanishing, i.e.
  \begin{align}
    \sup_\Pnullin \Var_\P \left ( \{ \widehat \mu_n^x(Z) - \mu^x(Z) \} \cdot \xi^y \right )  = o \left ( 1 \right ) ~~ \text{and}~~\sup_\Pnullin \Var_\P \left ( \{ \widehat \mu_n^y(Z) - \mu^y(Z) \} \cdot \xi^x \right )  = o \left ( 1 \right ).
  \end{align}
\end{assumptionGCM}

\begin{assumptionGCM}[$\Pcal_0$-uniformly bounded moments]
  \label{assumption:GCM-finite-moments}
  The true product residuals defined above have $\Pcal_0$-uniformly upper-bounded $(2+\delta)^\tth$ moments for some $\delta > 0$ and uniformly lower-bounded second moments:
  \begin{align}
    \sup_{\P \in \Pcal_0} \EE_\P \left \lvert \xi^x \xi^y \right \rvert^{2+\delta} < \infty \quad\text{and}\quad \inf_{\P \in \Pcal_0} \Var_\P ( \xi^x \xi^y ) > 0.
  \end{align}
\end{assumptionGCM}

With these three assumptions in mind, we are ready to recall a simplified version of \citet[Theorem 6]{shah2020hardness}.
\begin{theorem}[\citetalias{shah2020hardness}: $\Pcal_0$-uniform validity of the GCM test]\label{theorem:shah-peters-gcm}
  Suppose $(X_i, Y_i, Z_i)_{i=1}^n$ are $\RR \times \RR \times \RR^d$-valued random variables. Let $\Pcal_0\subset \Pcal$ be the collection of distributions in $\Pcal$ satisfying the conditional independence null $H_0$ and \namecref{assumption:GCM-finite-moments}s~\ref{assumption:GCM-regression-product-errors}--\ref{assumption:GCM-finite-moments}. Then,
  \begin{equation}\label{eq:shah-peters-type-I-err}
    \lim_{n \to \infty}\sup_{\P \in \Pnull} \sup_{x \in \RR} \left \lvert  \P(\sqrt{n}\GCMSP_n \leq x) - \Phi(x) \right \rvert = 0.
  \end{equation}
and hence $\Gamma_k^\brackm := \1 \left \{ |\sqrt{n}\GCMSP_n| \geq \Phi^{-1}(1-\alpha/2) \right \}$ is a $\Pnull$-uniform level-$\alpha$ test. 
\end{theorem}

We will now shift our focus to \emph{sequential} conditional independence testing with anytime-valid type-I error guarantees. Before deriving an explicit test, we first demonstrate in \cref{theorem:hardness} that the hardness of conditional independence testing highlighted in \eqref{eq:best-case-power-worst-case-type-i-err} has a similar analogue in the anytime-valid regime.

\subsection{On the hardness of anytime-valid conditional independence testing}
As mentioned in \cref{section:cit-refresher}, \citetalias{shah2020hardness} illustrated the fundamental hardness of conditional independence testing by showing that unless additional restrictions are placed on the null hypothesis $\Pcal_0^\star$, any $\Pcal_0^\star$-uniformly valid (fixed-$n$) test is powerless against any alternative (recall \eqref{eq:best-case-power-worst-case-type-i-err}).
Does an analogous result hold if $\dot \Gamma_n$ is replaced by an anytime-valid hypothesis test $\widebar \Gamma_k^\brackm$ as in \cref{definition:distribution-time-uniform-type-I-error-control}? The following theorem gives an answer to this question, confirming that anytime-valid conditional independence testing is fundamentally hard in a sense similar to \eqref{eq:best-case-power-worst-case-type-i-err}. 

\begin{theorem}[Hardness of anytime-valid conditional independence testing]\label{theorem:hardness}
  Suppose $\infseqn{X_n, Y_n, Z_n}$ are $[0, 1]^3$-valued triplets on the probability spaces $(\Omega, \Fcal, \Pcal^\star)$ where $\Pcal^\star$ consists of all distributions supported on $[0, 1]^3$. Let $\Pcal_0^\star \subseteq \Pcal^\star$ be the subset of distributions satisfying the conditional independence null $H_0$ and denote $\Pcal_1^\star := \Pcal^\star \setminus \Pcal_0^\star$. Then for any potentially randomized test $\infseqkm{\widebar \Gamma_k^\brackm}$,
  \begin{equation}\label{eq:hardness-anytime}
    \sup_{\P \in \Pcal_1^\star} \limsup_{m \to \infty} \P \left ( \exists k \geq m : \widebar \Gamma_k^\brackm = 1 \right ) \leq \limsup_{m \to \infty} \sup_{\P \in \Pcal_0^\star} \P \left ( \exists k \geq m : \widebar \Gamma_k^\brackm = 1 \right ).
  \end{equation}
  In other words, no $\Pcal_0^\star$-uniform anytime-valid test can have power against any alternative in $\Pcal_1^\star$ at any $\{m, m+1, \dots \}$-valued stopping time no matter how large $m$ is.
\end{theorem}

The proof can be found in \cref{proof:hardness}. It should be noted that \cref{theorem:hardness} is not a corollary of \citetalias{shah2020hardness}'s fixed-$n$ hardness result in \eqref{eq:best-case-power-worst-case-type-i-err} since while it is true that the time-uniform \emph{type-I error} in the right-hand side of \eqref{eq:hardness-anytime} is always larger than its fixed-$n$ counterpart, the time-uniform \emph{power} in the left-hand side of \eqref{eq:hardness-anytime} is typically much larger than the fixed-$n$ power. Indeed, while an important facet of hypothesis testing is to find tests with high power, the time-uniform power of anytime-valid tests is typically \emph{equal to} 1, and such tests are sometimes referred to explicitly as ``tests of power 1'' for this reason \citep{robbins1974expected}. This should not be surprising since the ability to reject at any stopping time (data-dependent sample size) larger than $m$ introduces a great deal of flexibility. The fact that this flexibility is insufficient to overcome $\Pcal_0^\star$-uniform control of the time-uniform type-I error is what makes \cref{theorem:hardness} nontrivial.

Using the techniques of \cref{section:distribution-uniform-anytime-valid-inference}, we will now derive an anytime-valid analogue of \citetalias{shah2020hardness}'s GCM test with similar distribution-uniform guarantees, allowing the tests and $p$-values to be continuously monitored and adaptively stopped.

\subsection{SeqGCM:~The sequential generalized covariance measure test}\label{section:SeqGCM}

We will now lay out the assumptions required for our SeqGCM test to have distribution-uniform anytime-validity. Similar to our discussion of the batch GCM test in the previous section, we will assume that for each $n$, $\widehat \mu_n^x$ and $\widehat \mu_n^y$ are trained from an independent sample. This can be achieved easily by supposing that at each time $n$, we observe pairs $(X_1^{(n)}, Y_1^{(n)}, Z_1^{(n)}), (X_2^{(n)}, Y_2^{(n)}, Z_2^{(n)})$ where the first is used for training $(\widehat \mu_i^x, \widehat \mu_i^y)_{i=n}^\infty$ and the second is used for evaluating $\{X_n - \widehat \mu_n^x(Z_n) \} \cdot \{Y_n - \widehat \mu_n^y(Z_n)\}$.

Recall that in \citetalias{shah2020hardness}'s GCM test, the test statistic $\GCMSP_n := \frac{1}{n} \sum_{i=1}^n R_{i,n} / \widehat \sigma_n^2$ was built from the product residuals $R_{i,n}$ that were defined in \eqref{eq:shah-peters-residuals}.
In particular, note that the regression estimators $\widehat \mu_n^x$ and $\widehat \mu_n^y$ are trained \emph{once} on a held-out sample of size $n$ and then evaluated on $Z_1, \dots, Z_n$, which is natural in the batch setting. By contrast, we will evaluate the product residual
\begin{equation}\label{eq:our-residuals}
  R_n := \left \{ X_n - \widehat \mu_{n}^x(Z_n) \right \} \left \{ Y_n - \widehat \mu_{n}^y(Z_n) \right \}
\end{equation}
to arrive at the test statistic
\begin{equation}
  \GCMWS_n := \frac{1}{n\widehat \sigma_{n}} \sum_{i=1}^n R_i,
\end{equation}
where we will abuse notation slightly and redefine $\widehat \sigma_n^2 := \frac{1}{n} \sum_{i=1}^n R_i^2 - \left ( \frac{1}{n} \sum_{i=1}^n R_i \right )^2$. The main difference between \eqref{eq:shah-peters-residuals} and \eqref{eq:our-residuals} is that in the latter case, the index for regression estimators $(\widehat \mu_n^x, \widehat \mu_n^y)$ is the same as those on which these functions are evaluated. Notice that while $\GCMWS_n$ is more amenable to online updates than $\GCMSP_n$, it does less to exploit the most up-to-date regression estimates. Nevertheless, as we will see shortly, it is still possible to control the distribution- and time-uniform asymptotic behavior of $\GCMWS_n$ under \emph{weak} regression consistency conditions on $(\widehat \mu_n^x, \widehat \mu_n^y)$. This is in contrast to \citet[Section 3]{waudby2021time} that also considered asymptotic time-uniform inference with nuisance estimation (focusing on the problem of average treatment effect estimation), but relied on \emph{strong} regression consistency conditions. It should be noted that the weak consistency rates we impose here are polylogarithmically faster than those considered by \citet{waudby2021time}. The key technique that will allow us to derive \emph{strong} convergence behavior of certain sample averages of nuisances from \emph{weak} consistency of regression functions is a distribution-uniform strong law of large numbers due to \citet*[Theorem 2]{waudby2024distribution}. This will be discussed further after the statement of \cref{theorem:seq-GCM}.

Since the assumptions required for our SeqGCM test are similar in spirit to those of \citetalias{shah2020hardness}'s batch GCM test (\namecref{assumption:GCM-finite-moments}s~\ref{assumption:GCM-regression-product-errors}, \ref{assumption:GCM-regression-estimators-do-not-diverge}, and~\ref{assumption:GCM-finite-moments}) we correspondingly name them ``\namecref{assumption:SeqGCM-regularity}s~\ref{assumption:SeqGCM-product-errors} and~\ref{assumption:SeqGCM-regularity}'' and \uline{underline} certain keywords to highlight their differences (we do not need to make additional moment assumptions beyond those found in \cref{assumption:GCM-finite-moments}, and thus there is no ``SeqGCM-3'' to introduce).

\begin{assumptionSeqGCM}[Product regression error decay]
  \label{assumption:SeqGCM-product-errors}
The product of weak convergence rates of $(\widehat \mu_n^x, \widehat \mu_n^y)$ is no slower than \uline{$(n\log^{2+\delta} n)^{-1/2}$} for some $\delta > 0$, i.e.
  \begin{equation}
    \exists \delta > 0:\quad \sup_\Pnullin \sqrt{\EE_\P \| \widehat \mu_n^x - \mu^x \|^2} \cdot \sqrt{\EE_\P \| \widehat \mu_n^y - \mu^y \|^2} = O \left ( \frac{1}{\sqrt{n \log^{2+\delta}(n)}} \right ).
  \end{equation}
\end{assumptionSeqGCM}

\begin{assumptionSeqGCM}[$\Pnull$-uniform regularity of regression errors]
  \label{assumption:SeqGCM-regularity}
  Both $\Var \left ( \{ \widehat \mu_n^x(Z) - \mu^x(Z) \} \cdot \xi^y \right )$ and $\Var \left ( \{ \widehat \mu_n^y(Z) - \mu^y(Z) \} \cdot \xi^x \right )$ are $\Pnull$-uniformly vanishing to 0 no slower than $1/(\log n)^{2+\delta}$ for some $\delta > 0$, i.e.
  \begin{align}
    &\sup_\Pnullin \Var_\P \left ( \{ \widehat \mu_n^x(Z) - \mu^x(Z) \} \cdot \xi^y \right )  = O \left ( \frac{1}{(\log n)^{2+\delta}}\right )\\
    \text{and}\quad&\sup_\Pnullin \Var_\P \left ( \{ \widehat \mu_n^y(Z) - \mu^y(Z) \} \cdot \xi^x \right )  = O \left ( \frac{1}{(\log n)^{2+\delta}}\right ).
  \end{align}
\end{assumptionSeqGCM}
\begin{figure}[!htb]
  \centering
  \includegraphics[width=0.49\textwidth]{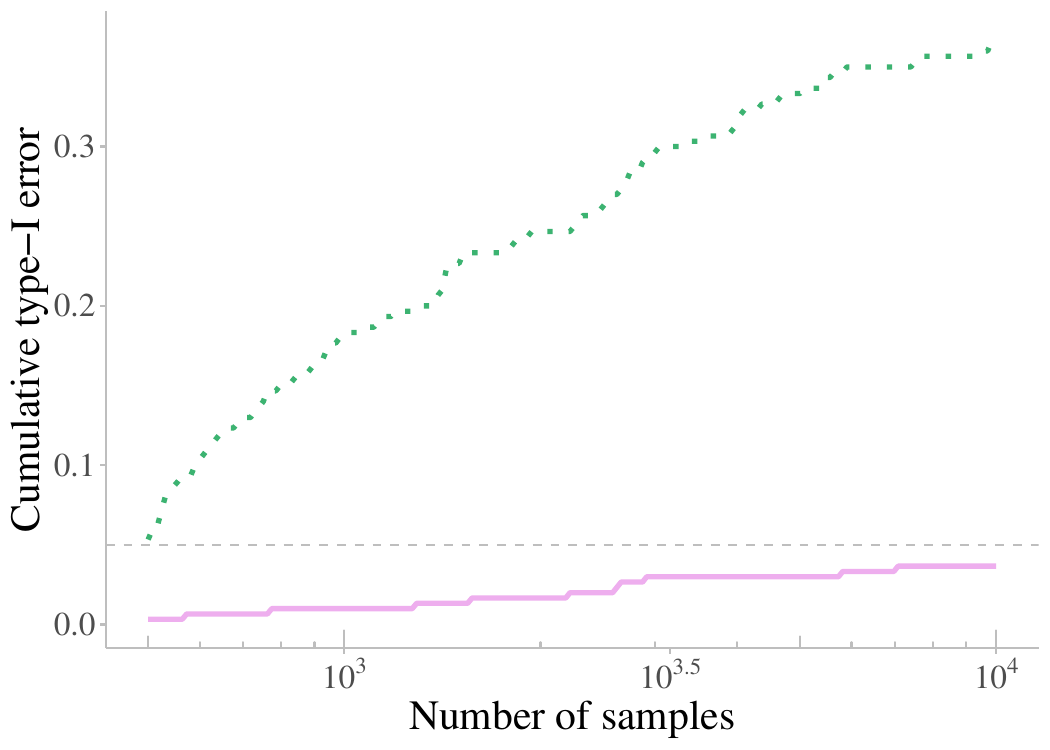}
  \includegraphics[width=0.49\textwidth]{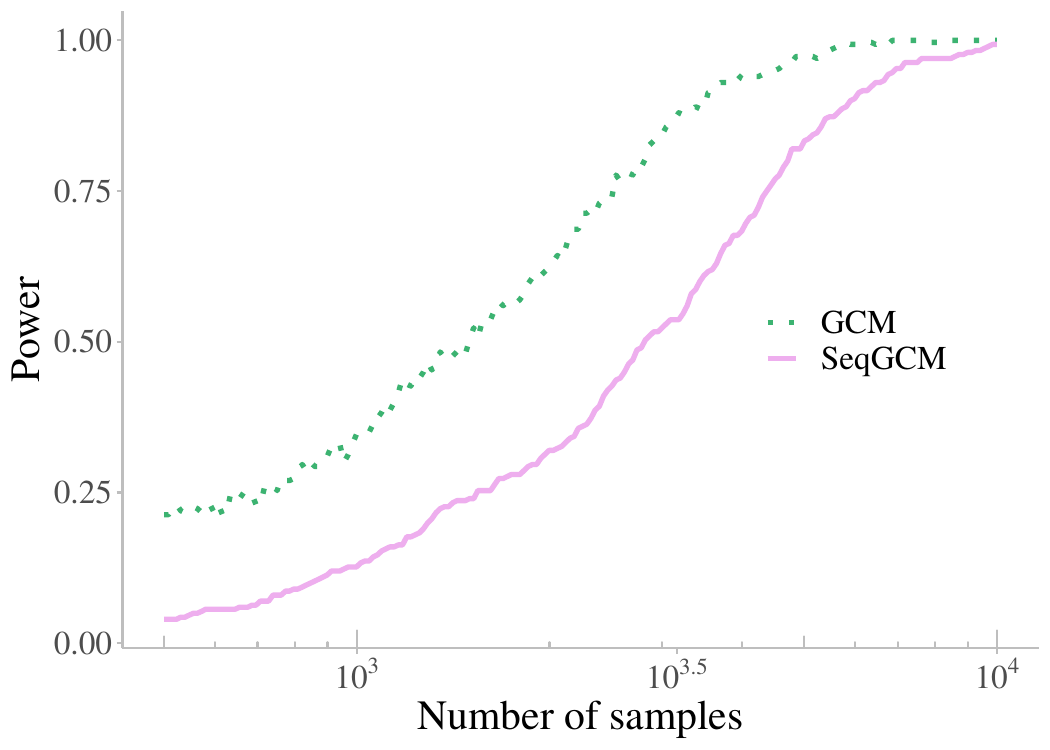}
  \caption{Empirical cumulative type-I error rates and power for the fixed-$n$ GCM test of \citetalias{shah2020hardness} versus the sequential GCM test (SeqGCM) in \cref{theorem:seq-GCM} with a target type-I error of $\alpha = 0.05$ in a simulated conditional independence testing problem.}
  \label{fig:seqcit}
\end{figure}
With Assumptions~\ref{assumption:SeqGCM-product-errors},~\ref{assumption:SeqGCM-regularity}, and~\ref{assumption:GCM-finite-moments} in mind, we are ready to state the $\Pcal_0$-uniform type-I error guarantees of the SeqGCM test.
\begin{theorem}[$\Pcal_0$-uniform type-I error control of the SeqGCM]\label{theorem:seq-GCM}
  Suppose $(X_i, Y_i, Z_i)_{i=1}^\infty$ are $\RR \times \RR \times \RR^d$-valued triplets and let $\Pcal_0 \subseteq \Pcal$ the subset of distributions in $\Pcal$ satisfying the conditional independence null $H_0$ and \namecref{assumption:SeqGCM-regularity}~\ref{assumption:SeqGCM-product-errors},~\ref{assumption:SeqGCM-regularity}, and~\ref{assumption:GCM-finite-moments}.
  Define
  \begin{equation}
    \widebar p_\km^\GCM := 1-\Psi \left (k(\GCMWS_k)^2 - \log(k / m) \right ).
  \end{equation}
  Then $\infseqkm{\widebar p_\km^\GCM}$ forms a $\Pnull$-uniform anytime $p$-value for the conditional independence null:
  \begin{equation}
    \lim_{m \to \infty} \sup_{\P \in \Pcal_0} \sup_{\alpha \in (0, 1)} \left \lvert \P \left ( \exists k \geq m : \widebar p_\km^\GCM \leq \alpha \right ) - \alpha \right \rvert = 0.
  \end{equation}
\end{theorem}

The proof can be found in \cref{proof:seq-GCM} and uses the results from the previous sections combined with the distribution-uniform strong laws of large numbers (SLLNs) for independent but non-identically distributed random variables due to \citet*[Theorem 2]{waudby2024distribution}. The latter is crucial to analyzing the (uniform) almost sure convergence properties of sample averages with online regression estimators under weak consistency assumptions (\ref{assumption:SeqGCM-product-errors} and \ref{assumption:SeqGCM-regularity}).

The left-hand side plot of \cref{fig:seqcit} demonstrates how the SeqGCM test controls the type-I error rate under the null uniformly over time while the standard GCM test fails to. The right-hand side plot compares their empirical power under one alternative. Notice that in the left-hand side plot, the type-I error rate for the GCM starts at around $\alpha = 0.05$ but steadily grows as more samples are collected. By contrast, the SeqGCM test remains below $\alpha = 0.05$ for all $k \geq m = 300$. In the right-hand side plot, we see that the power of the GCM test is higher than that of SeqGCM. This is unsurprising given that SeqGCM has a stronger (time-uniform) type-I error guarantee, but both have power near 1 after 10,000 samples.

To give some intuition as to when \cref{assumption:SeqGCM-product-errors} may be satisfied, suppose that $\mu^x$ and $\mu^y$ are $d$-dimensional and H\"older $s$-smooth \citep[\S 3.2]{gyorfi2002distribution}. Note that the minimax rate for estimating such functions in the resulting class of distributions $\Pcal(s)$ is given by
\begin{equation}
 \inf_{\widehat \mu_n^x} \sup_{\P \in \Pcal(s)} \EE_\P \| \widehat \mu_n^x - \mu^x \|^2 \asymp n^{-2s/(2s + d)},
\end{equation}
and similarly for $\mu^y$. In particular, if $d < 2s$ so that the dimension is not too large relative to the smoothness, then minimax-optimal local polynomial estimators $\widehat \mu^x_n$ and $\widehat \mu_n^y$ for $\mu^x$ and $\mu^y$ can be constructed and will be $\Pcal(s)$-weakly consistent at rates of $o \left ( (n \log^{2+\delta} n)^{-1/4} \right )$. In this case, \cref{assumption:SeqGCM-product-errors} (and \cref{assumption:GCM-regression-product-errors}) will be satisfied as long as $ \Pnull \subseteq \Pcal(s)$. More broadly, any regression algorithms can be used to construct $\widehat \mu_n^x$ and $\widehat \mu_n^y$ (e.g.~using random forests, neural networks, nearest neighbors, etc.) and they can further be selected via cross-validation or aggregated \citep{breiman1996stacked,tsybakov2003optimal}.

\section{Summary \& discussion}

In this paper, we developed a theory of distribution-uniform anytime-valid inference by relating it to quantile- and distribution-uniform convergence in distribution of suprema of Gaussian processes. This theory easily led to the derivation of methodological insights, resulting in one- and two-sided anytime-valid inference procedures for means of \iid{} random variables, a general problem in statistical inference that is a central element of many more complex settings. As one illustration, we use the outlined methods to derive a sequential test for conditional independence that does not rely on the Model-X assumption, and this seems to be the first of its kind. Zooming out, the results of this paper open vistas for several other settings in anytime-valid inference, such as distribution-uniform methods under martingale dependence, high-dimensional settings, and Berry-Esseen-type bounds. We leave these nontrivial extensions for future work.

\paragraph{Acknowledgments.} IW-S thanks Tudor Manole, Rajen Shah, and Kenta Takatsu for insightful discussions. The authors acknowledge support from NSF grants IIS-2229881 and DMS-2310718.

\bibliographystyle{plainnat}
\bibliography{references.bib}
\newpage
\appendix

\section{Proofs of the main results}

\subsection{Proof of \cref{theorem:consistent-variance-estimation}}\label{proof:consistent-variance-estimation}

\begin{proof}
Fix $\eps > 0$. We will rely on the result of \citet[Corollary 1]{waudby2024distribution} which states that for a fixed $\gamma > 0$, if 
\begin{equation}
  \lim_\mto \sup_\Pin  \EE_\P \left [ |X_1- \EE_\P[X_1]|^{2+\gamma} \1 \{ |X_1 - \EE_\P[X_1]|^{2+\gamma} \geq m \} \right ] = 0, 
\end{equation}
then it holds that
\begin{equation}\label{eq:variance estimation wsetal result}
  \forall s > 0, \quad \lim_\mto \sup_\Pin \P \left ( \supkm \left \{ \frac{|\widehat \sigma_k^2 - \sigma_\P^2|}{k^{2/(2+\gamma)-1}} \right \} \geq s \right ) = 0.
\end{equation}
By the de la Vall\'ee-Poussin criterion of uniform integrability (see \citep{chong1979theorem,hu2011note,chandra2015vallee} or \citet[Appendix A]{waudby2024distribution}), the above holds
with $\gamma = \delta / 2$ under \cref{assumption:upper bound on variance}. Now,
let $\ubar \sigma^2 := \inf_\Pin \Var_\P(X_1)$. Then for any $m \in \NN$ and $\Pin$,
\begin{align}
  \P \left ( \supkm k^{1-\frac{2}{2+\delta/2}} \left ( \frac{\widehat \sigma_k^2}{\sigma_\P^2} - 1 \right ) \geq \eps \right ) &= \P \left ( \supkm \frac{k^{1-\frac{2}{2+\delta/2}}}{\sigma_\P^2} \left ( \widehat \sigma_k^2 - \widehat \sigma_\P^2 \right ) \geq \eps \right )\\
  &\leq \P \left ( \supkm k^{1-\frac{2}{2+\delta/2}} \left ( \widehat \sigma_k^2 - \widehat \sigma_\P^2 \right ) \geq \ubar \sigma^2 \eps \right ).
\end{align}
Taking suprema over $\Pin$ and limits as $\mto$ while relying on \eqref{eq:variance estimation wsetal result}, we have that
\begin{equation}
  \lim_\mto \sup_\Pin\P \left ( \supkm k^{1-\frac{2}{2+\delta/2}} \left ( \frac{\widehat \sigma_k^2}{\sigma_\P^2} - 1 \right ) \geq \eps \right ) = 0,
\end{equation}
which completes the proof.
\end{proof}

\subsection{Proof of \cref{proposition:oPcalas-calculus}}\label{proof:oPcalas-calculus}

\begin{proof}[\proofpreamble{}of \eqref{eq:oPcalas-implies-OPcalas}]
Suppose that $X_n = \oPcalas(a_n)$. We want to show that for any $\delta$, there exists $C \equiv C(\delta)$ and $M \equiv M(\delta)$ so that for all $m \geq M$,
\begin{equation}
  \sup_\Pin \P \left ( \sup_{k \geq m} |a_k^{-1} X_k | \geq C \right ) < \delta.
\end{equation}
This is immediate from the definition of $\oPcalas(a_n)$. Indeed, fix any $\eps > 0$ and choose $M\equiv M(\eps)$ so that for any $m \geq M$,
\begin{equation}
  \sup_\Pin \P \left ( \sup_{k \geq m} |a_k^{-1} X_k | \geq \eps \right ) < \delta.
\end{equation}
Identifying $C$ with $\eps$ completes the proof.
\end{proof}

\begin{proof}[\proofpreamble{}of \eqref{eq:product-is-oPcalas}]
Suppose that $X_n = A_n B_n$ with $A_n = \oPcalas(a_n)$ and $B_n = \oPcalas(b_n)$. We want to show that $a_n^{-1}b_n^{-1} X_n = \oPcalas(1)$. More formally, our goal is to show that for arbitrary $\eps, \delta > 0$, there exists $M\equiv M(\eps, \delta) \geq 1$ so that for all $m \geq M$,
\begin{equation}\label{eq:proof-product-of-oPcalas-goal}
  \sup_\Pin \P \left ( \exists k \geq m : |a_k^{-1} b_k^{-1} X_k| \geq \eps \right ) < \delta.
\end{equation}
Choose $M$ sufficiently large so that for all $m \geq M$,
\begin{equation}
  \sup_\Pin \P \left ( \sup_{k \geq m} |a_k^{-1}A_k| \geq \sqrt{\eps/2} \right ) < \delta\quad\text{and}\quad\sup_\Pin \P \left ( \sup_{k \geq m} |b_k^{-1}B_k| \geq \sqrt{\eps/2} \right ) < \delta.
\end{equation}
Then, writing out the equation in \eqref{eq:proof-product-of-oPcalas-goal}, we have that
\begin{align}
  &\sup_\Pin \P \left ( \exists k \geq m : |a_k^{-1} b_k^{-1} X_k| \geq \eps \right ) \\
  \leq\ &\sup_\Pin \P \left ( \exists k \geq m : |a_k^{-1}A_k| |b_k^{-1} B_k| \geq \eps \right )\\
  \leq\ &\sup_\Pin \P \left ( \exists k \geq m : |a_k^{-1}A_k| |b_k^{-1} B_k| \geq \eps \bigm \vert \sup_{k \geq m}|a_k^{-1}A_k| < \sqrt{\eps / 2}\text{ and } \sup_{k \geq m}|b_k^{-1}B_k| < \sqrt{\eps / 2} \right ) + \\
  &\sup_\Pin \P \left ( \sup_{k \geq m}|a_k^{-1}A_k| < \sqrt{\eps/2}\text{ and } \sup_{k \geq m}|b_k^{-1}B_k| < \sqrt{\eps / 2} \right ) \\
  \leq\ &\underbrace{\sup_\Pin \P \left ( \exists k \geq m : \eps / 2 \geq \eps \right )}_{= 0} + \\
  &\underbrace{\max \left \{\sup_\Pin \P \left ( \sup_{k \geq m}|a_k^{-1}A_k| < \sqrt{\eps/2} \right ),\ \P \left (\sup_{k \geq m}|b_k^{-1}B_k| < \sqrt{\eps / 2} \right )  \right \}}_{< \delta}\\
  <\ &\delta,
\end{align}
which completes the proof.
\end{proof}

\begin{proof}[\proofpreamble{}of \eqref{eq:product-is-OPcalas}]
 Suppose that $X_n = A_n B_n$ with $A_n = \OPcalas(a_n)$ and $B_n = \OPcalas(b_n)$. Our goal is to show that for any $\delta > 0$, there exists some $C \equiv C(\delta)$ and $M\equiv M(C, \delta)$ so that
\begin{equation}
  \sup_\Pin \P \left ( \sup_{k \geq m} |a_n^{-1}b_n^{-1}X_n| > C \right ) < \delta.
\end{equation}
  Fix $\delta > 0$. Let $C_a, M_a, C_b, M_b$ be sufficiently large so that for all $m \geq \max\{M_a, M_b\}$,
  \begin{equation}
    \sup_\Pin \P \left ( \sup_{k \geq m} |a_k^{-1} A_k| \geq M_a \right ) < \delta\quad\text{and}\quad\sup_\Pin \P \left ( \sup_{k \geq m} |b_k^{-1} B_k| \geq M_b \right ) < \delta.
  \end{equation}
  Now, set $C = C_a C_b + 1$. Then,
  \begin{align}
    &\sup_\Pin \P \left ( \sup_{k \geq m} |a_k^{-1}b_k^{-1}X_k| \geq C \right ) \\
    \leq\ &\sup_\Pin \P \left ( \sup_{k \geq m} |a_k^{-1}A_k | |b_k^{-1}B_k| \geq C \right )\\
    \leq\ &\sup_\Pin \P \left ( \sup_{k \geq m} C_a C_b \geq C \right ) + \sup_\Pin \P \left ( \sup_{k \geq m} |a_k^{-1} A_k | > C_a\text{ and } |b_k^{-1}B_k| > C_b\right )\\
    \leq\ &\underbrace{\sup_\Pin \P \left ( \sup_{k \geq m} C_a C_b \geq C_a C_b +1 \right )}_{= 0} +\\
    &\underbrace{\max\left \{ \sup_\Pin \P \left ( \sup_{k \geq m} |a_k^{-1} A_k | > C_a \right ),\ \sup_\Pin \P \left ( \sup_{k\geq m}|b_k^{-1}B_k| > C_b\right ) \right \}}_{ < \delta},
  \end{align}
  which completes the proof.
\end{proof}

\begin{proof}[\proofpreamble{}of \eqref{eq:sum-is-OPcalas}]
Suppose $X_n = A_n + A_n'$ with both $A_n = \oPcalas(a_n)$ and $A_n' = \OPcalas(a_n)$. The goal is to show that for every $\delta > 0$, there exists $C > 0$ and $M \geq 1$ so that for all $m \geq M$,
\begin{equation}
  \sup_\Pin \P \left ( \sup_{k\geq m} a_k^{-1} |X_k| > C \right ) < \delta.
\end{equation}
  Fix $\delta > 0$. Let $C'$ and $M'$ be so that $\sup_\Pin \P \left ( \sup_{k\geq m} a_k^{-1} |A_k'| > C' \right ) < \delta/2$. Fix any $\eps \in (0, C')$ and let $M^\star$ be so that $\sup_\Pin \P \left ( \sup_{k \geq m} a_k^{-1} |A_k| \geq \eps \right ) < \delta/2$ for all $m \geq M^\star$. Choose $M > \max\{M', M^\star\}$. Then, for all $m \geq M$,
  \begin{align}
   &\sup_\Pin \P \left ( \sup_{k \geq m} a_k^{-1} |A_k + A_k'| \geq C' \right ) \\
    \leq \ &\sup_\Pin \P \left ( \sup_{k \geq m} a_k^{-1} |A_k| + a_k|A_k'| \geq C' \right )\\
    \leq \ &\sup_\Pin \P \left ( \sup_{k \geq m} a_k^{-1} |A_k| \geq C' \right ) + \sup_\Pin \P \left ( \sup_{k \geq m} a_k^{-1} |A_k'| \geq C' \right )\\
    \leq \ &\sup_\Pin \P \left ( \sup_{k \geq m} a_k^{-1} |A_k| \geq \eps \right ) + \sup_\Pin \P \left ( \sup_{k \geq m} a_k^{-1} |A_k'| \geq C' \right )\\
    < \ &\delta,
  \end{align}
  which completes the proof.
\end{proof}

\begin{proof}[\proofpreamble{}of \eqref{eq:sum-of-oPcalas}]
Suppose $X_n = A_n + B_n$ with $A_n = \oPcalas(a_n)$ and $B_n = \oPcalas(b_n)$. The goal is to show that for every $\eps, \delta > 0$, there exists $M \geq 1$ so that for all $m \geq M$,
\begin{equation}
  \sup_\Pin \P \left ( \sup_{k\geq m} c_k^{-1} |X_k| \geq \eps \right ) < \delta,
\end{equation}
where $c_k = \max\{a_k, b_k\}$.
  Fix $\eps, \delta > 0$. Let $M$ be so that $\sup_\Pin \P \left ( \sup_{k\geq m} a_k |A_k| > \eps \right ) < \delta/2$ and $\sup_\Pin \P \left ( \sup_{k\geq m} b_k |B_k| > \eps \right ) < \delta/2$ for all $m \geq M$. Then, for all $m \geq M$,
  \begin{align}
   &\sup_\Pin \P \left ( \sup_{k \geq m} c_k^{-1} |A_k + B_k| \geq \eps \right ) \\
    \leq \ &\sup_\Pin \P \left ( \sup_{k \geq m} c_k^{-1} |A_k| + c_k^{-1}|B_k| \geq \eps \right )\\
    \leq \ &\sup_\Pin \P \left ( \sup_{k \geq m} a_k^{-1} |A_k| + b_k^{-1}|B_k| \geq \eps \right )\\
    \leq \ &\sup_\Pin \P \left ( \sup_{k \geq m} a_k^{-1} |A_k| \geq \eps \right ) + \sup_\Pin \P \left ( \sup_{k \geq m} b_k^{-1} |B_k| \geq \eps \right )\\
    <\ &\delta,
  \end{align}
  which completes the proof.
\end{proof}

\begin{proof}[\proofpreamble{}of \eqref{eq:sum-of-OPcalas}]
  Fix $\delta > 0$. Let $A_n = \OPcalas(a_n)$ and $B_n = \OPcalas(b_n)$. Let $C$, and $M$ be such that $\sup_\Pin \P \left ( \exists k \geq m : a_k^{-1}| A_k | > C/2 \right ) \leq \delta/2$ and $\P \left ( \exists k \geq m : b_k^{-1} |B_k| > C /2\right ) < \delta/2$ for all $m \geq M$. Let $c_n := \max\{a_n, b_n\}$. Then,
  \begin{align}
    \P \left ( \exists k \geq m: c_k^{-1} |A_k + B_k | > C \right ) &\leq \P \left ( \exists k \geq m: c_k^{-1} |A_k| + c_k^{-1}|B_k | > C \right )\\
                                                                    &\leq\P \left ( \exists k \geq m: c_k^{-1} |A_k| > \frac{C}{2} \right ) + \PP \left ( \exists k \geq m : c_k^{-1}|B_k | > \frac{C}{2} \right ) \\
                                                                    &\leq\P \left ( \exists k \geq m: a_k^{-1} |A_k| > \frac{C}{2} \right ) + \P \left ( \exists k \geq m : b_k^{-1}|B_k | > \frac{C}{2} \right ) \\
    &\leq \frac{\delta}{2} + \frac{\delta}{2} = \delta.
  \end{align}
  Since $\delta$ was arbitrary, this completes the proof.
\end{proof}

\begin{proof}[\proofpreamble{}of \eqref{eq:slightly-faster-convergence}]
Let $\eps, \delta > 0$. The goal is to show that there exists $M \equiv M(\eps, \delta) \geq 1$ so that for all $m \geq M$,
\begin{equation}
  \sup_\Pin \P \left ( \sup_{k \geq m} b_k^{-1} |X_k| \geq \eps \right ) < \delta.
\end{equation}
  Let $C > 0$ and $M_1 \geq 1$ be constants so that $\sup_\Pin \P \left ( \sup_{k \geq m} a_k^{-1}|X_k| \geq C \right ) < \delta$ for all $m \geq M_1$. Moreover, choose $M_2 \geq 1$ so that $a_k / b_k < \eps/C$ for all $k \geq M_2$. Set $M := \max\{M_1, M_2\}$. Then, for all $m \geq M$,
  \begin{align}
    &\sup_\Pin \P \left ( \sup_{k \geq m} b_k^{-1} |X_k| \geq \eps \right ) \\
    =\ &\sup_\Pin \P \left ( \exists k \geq m: b_k^{-1} |X_k| \geq \eps \right ) \\
    =\ &\sup_\Pin \P \left ( \exists k \geq m: a_k^{-1} |X_k| \geq (b_k / a_k)\eps \right ) \\
    \leq\ &\sup_\Pin \P \left ( \exists k \geq m: a_k^{-1} |X_k| \geq (C / \cancel{\eps}) \cdot \cancel{\eps} \right ) \\
    =\ &\sup_\Pin \P \left ( \exists k \geq m: a_k^{-1} |X_k| \geq C \right ) \\
    <\ &\delta,
  \end{align}
  which completes the proof.
\end{proof}

\begin{proof}[\proofpreamble{}of \eqref{eq:reciprocal of 1 plus oPcalas}]
  Fix $\eps > 0$ and $\delta > 0$. Choose $\gamma > 0$ sufficiently small so that
  \begin{equation}
    \frac{1}{1-\gamma} \leq 1 + \eps\quad\text{and}\quad \frac{1}{1+\gamma}\geq 1-\eps.
  \end{equation}
  Choose $M$ so that $\forall m \geq M$,
  \begin{equation}
    \sup_\Pin \P \left ( \supkm |X_k| \geq \gamma \right ) \leq \delta.
  \end{equation}
  Therefore, it holds that for all $m \geq M$,
  \begin{align}
    \sup_\Pin \P \left ( \supkm \abs{\frac{1}{1 + X_k}-1} \geq \eps \right ) &\leq \delta,
  \end{align}
  which completes the proof.
\end{proof}

\subsection{Proof of \cref{proposition:properties of psi}}\label{proof:properties of psi}
\begin{proof}
  Starting with the computation of $\psi$, we observe that for any $x \geq 0$,
  \begin{align}
    \psi(x) &:= \frac{\dd \Psi(x)}{\dd x} = \cancel{\frac{\phi(\sqrt{x})}{\sqrt{x}}} - \cancel{\frac{\phi(\sqrt{x})}{\sqrt{x}}} - \cancel{\sqrt{x}} \frac{\phi'(\sqrt{x})}{\cancel{\sqrt{x}}}\\
            &= -\phi'(\sqrt{x}) \\
    &= \sqrt{x} \phi(\sqrt{x}). 
  \end{align}
  It is easy to then check that $\sup_{x \geq 0} |\psi(x)| \leq (2\pi e)^{-1/2}$, and hence $\Psi$ is $(2\pi e)^{-1/2}$-Lipschitz.

  Moving on to the computation of $\psiplus$, observe that for any $x \geq 0$, we have
  \begin{align}
    \psiplus(x) &:= \frac{\dd \Psiplus(x)}{\dd x} \\
                &= \cancel{\frac{\phi(\sqrt{x})}{2 \sqrt{x}}} - \left ( \cancel{\frac{\phi(\sqrt{x}) }{2\sqrt{x}}} + \cancel{\sqrt{x}} \frac{\phi'(\sqrt{x})}{2\cancel{\sqrt{x}}}  \right ) - \left ( \frac{\cancel{2}\phi(\sqrt{x}) \phi'(\sqrt{x})}{\cancel{2} \sqrt{x} \Phi(\sqrt{x})} + \phi(\sqrt{x})^2 \left ( -\Phi(\sqrt{x})^{-2}  \frac{\phi(\sqrt{x})}{2\sqrt{x}} \right )  \right )\\
                &= \frac{\sqrt{x} \phi(\sqrt{x})}{2} + \frac{\phi(\sqrt{x})^2}{\Phi(\sqrt{x})} + \frac{\phi(\sqrt{x})^3 }{2 \Phi(\sqrt{x})^2 \sqrt{x}} \\
                &= \frac{\phi(\sqrt{x})}{2 \sqrt{x} \Phi(\sqrt{x})^2} \left ( x \Phi(\sqrt{x})^2 + 2 \sqrt{x} \phi(\sqrt{x}) \Phi(\sqrt{x}) + \phi(\sqrt{x})^2  \right ) \\
    &= \frac{\phi(\sqrt{x})}{2 \sqrt{x} \Phi(\sqrt{x})^2} \left ( \phi(\sqrt{x}) + \sqrt{x} \Phi(\sqrt{x}) \right )^2.
  \end{align}
  Now fix $a > 0$. We will show that $\psiplus$ is uniformly bounded on $[a, \infty)$. Clearly $\psiplus$ is positive on $[a, \infty)$. Notice that $\psiplus(a)$ is finite and that
  \begin{equation}
    \lim_\xto \psiplus(x) = 0.
  \end{equation}
  By the extreme value theorem, it follows that $\psiplus$ is uniformly bounded and hence $\Psiplus$ is Lipschitz continuous on $[a,\infty)$. This completes the proof.
\end{proof}

\subsection{Proof of \cref{proposition:robbins-siegmund from Gaussians}}\label{proof:robbins-siegmund from Gaussians}

\begin{proof}
  Beginning with $(i)$, we let $x \geq 0$ and note that
  \begin{align}
    \P \left ( \sup_{t \geq 1} \left \{ \frac{W(t)^2}{t} - \log (t) \right \} \geq x \right ) &= \P \left ( \exists t \geq 1 : \frac{W(t)^2}{t} - \log t \geq x^2 \right )\\
                        &= \P \left ( \exists t \geq 1 : |W(t)| \geq \sqrt{t \left [ x^2 + \log t \right ]} \right )\\
    &= 2[1-\Phi(x) + x\phi(x)] \equiv 1-\Psi(x^2),
  \end{align}
  where the last line follows from \citet[Example 3]{robbins1970boundary} but with their value of $\tau$ set to 1 (see also \citep[Lemma A.14]{waudby2021time}). This completes the proof of part $(i)$.

  Moving on to the claim in $(ii)$, we similarly have that for any $x \geq 0$,
  \begin{align}
    \P \left ( \sup_{t \geq 1} \left \{ g^\star \left ( \frac{W(t)}{\sqrt{t}} \right ) - \log (t) \right \} \geq x \right ) &= \P \left ( \exists t \geq 1 : W(t) \geq t^{1/2}  (g^\star)^{-1} \left ( x + \log (t) \right )  \right )\\
    &= 1-\Phi(x) + \phi(x) \left ( x + \frac{\phi(x)}{\Phi(x)} \right ),
  \end{align}
  which follows from the derivation preceding Example 3 in \citet{robbins1970boundary}.

  Moving on to the claims in $(iii)$ and $(iv)$, note that by \citet[Example 3]{robbins1970boundary}, we have that
  \begin{equation}
    \forall x\geq 0, \quad \lim_\mto \P \left ( \supkm \left \{ \frac{G_k^2}{k \sigma^2} - \log \left ( \frac{k}{m} \right )  \right \} \leq x \right ) = \Psi(x)
  \end{equation}
  and
  \begin{equation}
    \forall x \geq 0, \quad \P \left ( \supkm \left \{ g^\star \left ( \frac{G_k}{\sqrt{k \sigma^2}} \right ) - \log \left ( \frac{k}{m} \right ) \right \} \leq x \right ) = \Psiplus(x).
  \end{equation}
  The desired results follow after invoking \cref{lemma:quantile uniformity for free}. This completes the proof.
\end{proof}

\subsection{Proof of \cref{theorem:main-convergence-in-distribution}}\label{proof:main-convergence-in-distribution}

We prove a more general lemma here that yields \cref{theorem:main-convergence-in-distribution} as a corollary by taking $R_k = 0$ for each $k \in \NN$.

\begin{lemma}\label{theorem:generalized convergence}
  Let $\infseqn{X_n}$ be \iid{} random variables satisfying \cref{assumption:upper bound on variance}.
  Let $\widehat \sigma_n^2$ be an estimator of the variance satisfying $\widehat \sigma_n^2 / \sigma^2 - 1 = \oPcalas(1/\log n)$ and let $R_n$ be a term satisfying $R_n = \oPcalas(1)$. Then
  \begin{equation}\label{eq:main-theorem-two-sided-general}
    \lim_\mto \sup_\Pin \sup_{x \geq 0} \left \lvert \P \left ( \supkm \left \{ \frac{S_k^2}{\widehat \sigma_k^2 k} + R_k - \log \left ( \frac{k}{m} \right ) \right \} \geq x \right ) - [ 1 - \Psi(x) ] \right \rvert = 0.
  \end{equation}
  Letting $g(x) = \left ( x^2 + 2 \log \Phi(x) \right ) \lor 0$ for $x \in \RR$, we also have that
  \begin{equation}\label{eq:main-theorem-one-sided-general}
    \lim_\mto \sup_\Pin \sup_{x \geq a_1} \pospart{ \P \left ( \supkm \left \{ g \left ( \frac{S_k \lor 0}{\widehat \sigma_k \sqrt{k}}\right ) + R_k - \log \left ( \frac{k}{m} \right ) \right \} \geq x \right ) - [1-\Psiplus(x)] } = 0,
  \end{equation}
  where $a_1 = \Psiplus^{-1}(1/2)$.
\end{lemma}

We first prove the result for the two-sided case displayed in \eqref{eq:main-theorem-two-sided-general} and then later consider the one-sided case displayed in \eqref{eq:main-theorem-one-sided-general}.

\begin{proof}[\proofpreamble{}for the two-sided case in \eqref{eq:main-theorem-two-sided-general}]
  Let $\eps > 0$ be arbitrary. Using the fact that $\Psi$ is Lipschitz (see \cref{proposition:properties of psi}) and hence uniformly continuous, choose $\delta > 0$ sufficiently small so that for any $y,z \in \RR$ with $|y-z| \leq 3\delta$, it holds that
  \begin{equation}
    |\Psi(y) - \Psi(z)| < \eps.
  \end{equation}
  Moreover, choose $M_1$ sufficiently large so that for all $ m \geq M_1$,
  \begin{equation}
    \sup_\Pin \P \left ( \supkm |R_k| \geq \delta \right ) < \eps.
  \end{equation}

  \paragraph{Step I: Controlling errors induced from estimating the variance}
  Using \cref{theorem:consistent-variance-estimation}, there exists some $\beta > 0$ so that
  \begin{equation}
    \forall s > 0, \quad \lim_\mto \sup_\Pin \P \left ( \supkm k^\beta |\widehat \sigma_k^2 / \sigma_\P^2 - 1| \geq s \right ) = 0,
  \end{equation}
  and hence it holds that $\widehat \sigma_n^2 / \sigma_\P^2 = \oPcalas\left (\log^{-1}(n)\right )$. Indeed, take $M_2$ sufficiently large so that for all $m \geq M_2$,
  \begin{equation}
    \sup_\Pin \P \left ( \supkm \log(k) | \widehat \sigma_k^2 / \sigma_\P^2 - 1| \geq \delta \right ) < \eps.
  \end{equation}
  Writing out the probability of central interest, we have for any $\Pin$, any $x \in \RR$, and any $m \geq \max \{ M_1, M_2 \}$,
  \begin{align}
    \P \left ( \supkm \left \{ \frac{S_k^2}{\widehat \sigma_k^2 k} + R_k - \log \left ( \frac{k}{m} \right ) \right \} \geq x \right ) &< \P \left ( \supkm \left \{ \frac{S_k^2}{\sigma_\P^2 (1 - \delta / \log k) k} - \log \left ( \frac{k}{m} \right ) \right \} \geq x - \delta \right ) + 2\eps\\
    &\leq \P \left ( \supkm \left \{ \frac{S_k^2}{\sigma_\P^2 k} - \log \left ( \frac{k}{m} \right ) \right \} \geq x - 2\delta \right ) + 2\eps.
  \end{align}
  Similarly, we have that
  \begin{align}
    \P \left ( \supkm \left \{ \frac{S_k^2}{\widehat \sigma_k^2 k} + R_k - \log \left ( \frac{k}{m} \right ) \right \} \geq x \right ) &\geq \P \left ( \supkm \left \{ \frac{S_k^2}{\sigma_\P^2 (1 + \delta / \log k) k} - \log \left ( \frac{k}{m} \right ) \right \} \geq x + \delta \right )(1-2\eps) \\
    &\geq \P \left ( \supkm \left \{ \frac{S_k^2}{\sigma_\P^2 k} - \log \left ( \frac{k}{m} \right ) \right \} \geq x + 2\delta \right ) - 2\eps.
  \end{align}

  \paragraph{Step II: Controlling the strong Gaussian approximation error}
  Without loss of generality, let $(\Omega, \Fcal, \Pcal)$ be sufficiently rich so that they can contain Gaussian random variables. Recalling \cref{theorem:strong-gaussian-approx} \citep[Theorem 4]{waudby2025nonasymptotic}, there exists a sequence $\infseqn{Y_n}$ of \iid{} Gaussian random variables with mean zero and variance $\Var_\P(X_1)$ so that once we put $S_n := \sum_{i=1}^n X_i$ and $G_n := \sum_{i=1}^n Y_i$, it holds that 
  \begin{align}
    \forall s > 0,\quad \lim_\mto \sup_\Pin \P \left ( \supkm \left \lvert \frac{S_k - G_k}{k^{1/q}} \right \rvert \geq s \right )  = 0,
  \end{align}
  or in other words, $S_n = G_n + \oPcalas(n^{1/q})$.
  Appealing to \cref{proposition:oPcalas-calculus}, we have that
  \begin{equation}
    S_n^2 = \left ( G_n + \oPcalas(n^{1/q}) \right )^2 = G_n^2 + \oPcalas(n^{2/r})
  \end{equation}
  for some $r > 2$, and hence $S_n^2 = G_n^2 + \oPcalas(n)$.
  As such, take $M_3$ sufficiently large so that for all $m \geq M_3$,
  \begin{equation}
    \sup_\Pin \P \left ( \supkm \left \lvert \frac{S_k^2 - G_k^2}{k} \right \rvert \geq \ubar \sigma^2 \delta \right ) < \eps.
  \end{equation}
  It then follows that for any $m \geq M_3$, any $\Pin$, and $x\in\RR$,
  \begin{align}
    \P \left ( \supkm \left \{ \frac{S_k^2}{\sigma_\P^2 k} - \log \left ( \frac{k}{m} \right ) \right \} \geq x \right ) &< \P \left ( \exists k \geq m :  \frac{G_k^2}{\sigma_\P^2 k} + \frac{\ubar \sigma^2 \delta}{\sigma_\P^2} - \log \left ( \frac{k}{m} \right ) \geq x \right ) + \eps\\
    &\leq \P \left ( \exists k \geq m :  \frac{G_k^2}{\sigma_\P^2 k} - \log \left ( \frac{k}{m} \right ) \geq x - \delta \right ) + \eps.
  \end{align}
  Similarly, it holds that for any $m\geq M_3$, $\Pin$, and $x \in \RR$,
  \begin{align}
    \P \left ( \supkm \left \{ \frac{S_k^2}{\sigma_\P^2 k} - \log \left ( \frac{k}{m} \right ) \right \} \geq x \right ) &> \P \left ( \exists k \geq m :  \frac{G_k^2}{\sigma_\P^2 k} - \frac{\ubar \sigma^2 \delta}{\sigma_\P^2} - \log \left ( \frac{k}{m} \right ) \geq x \right )(1-\eps)\\
    &\geq \P \left ( \exists k \geq m :  \frac{G_k^2}{\sigma_\P^2 k} - \log \left ( \frac{k}{m} \right ) \geq x + \delta \right ) - \eps.
  \end{align}
  With Step I in mind, we have that for any $m \geq \max \{M_1, M_2, M_3\}$, any $\Pin$, and any $x \in \RR$,
  \begin{equation}
    \P \left ( \supkm \left \{ \frac{S_k^2}{\widehat \sigma_k^2k} - \log \left ( \frac{k}{m} \right ) \right \} \geq x \right ) < \P \left ( \supkm \left \{ \frac{G_k^2}{\sigma_\P^2 k} - \log \left ( \frac{k}{m} \right ) \right \} \geq x - 3\delta \right ) + 3\eps
  \end{equation}
  and 
  \begin{equation}
    \P \left ( \supkm \left \{ \frac{S_k^2}{\widehat \sigma_k^2k} - \log \left ( \frac{k}{m} \right ) \right \} \geq x \right ) > \P \left ( \supkm \left \{ \frac{G_k^2}{\sigma_\P^2 k} - \log \left ( \frac{k}{m} \right ) \right \} \geq x + 3\delta \right ) - 3\eps.
  \end{equation}

  \paragraph{Step III: Appealing to quantile-uniform convergence of transformed Gaussians}
  By \cref{proposition:robbins-siegmund from Gaussians}, take $M_4$ sufficiently large so that for all $m \geq M_4$, it holds that
  \begin{equation}
    \sup_{x \in \RR} \left \lvert \P \left ( \supkm \left \{ \frac{G_k^2}{\sigma_\P^2 k} - \log \left ( \frac{k}{m} \right ) \right \} \geq x \right ) - [1-\Psi(x)] \right \rvert < \eps.
  \end{equation}
  Further appealing to uniform continuity of $\Psi$ (\cref{proposition:properties of psi}) and the definition of $\delta$, we have for all $m \geq M_4$,
  \begin{equation}
   \P \left ( \supkm \left \{ \frac{G_k^2}{\sigma_\P^2 k} - \log \left ( \frac{k}{m} \right ) \right \} \geq x -3\delta \right ) - [1-\Psi(x)] < 2\eps,
  \end{equation}
  and through a similar argument,
  \begin{equation}
    \P \left ( \supkm \left \{ \frac{G_k^2}{\sigma_\P^2 k} - \log \left ( \frac{k}{m} \right ) \right \} \geq x -3\delta \right ) - [1-\Psi(x)] > -2\eps
  \end{equation}

  \paragraph{Step IV: Combining the previous steps}
  It now holds that for any $m \geq \max\{M_1,M_2,M_3, M_4\}$, any $\Pin$, and any $x \in \RR$, 
  \begin{align}
    &\P \left ( \supkm \left \{ \frac{S_k^2}{\widehat \sigma_k^2k} + R_k - \log \left ( \frac{k}{m} \right ) \right \} \geq x \right ) - [1-\Psi(x)] \\
    <\ &\P \left ( \supkm \left \{ \frac{S_k^2}{\sigma_\P^2 k} - \log \left ( \frac{k}{m} \right ) \right \} \geq x-2\delta \right ) - [1-\Psi(x)] + 2\eps \\
                                                                                                                                <\ &\P \left ( \supkm \left \{ \frac{G_k^2}{\sigma_\P^2 k} - \log \left ( \frac{k}{m} \right ) \right \} \geq x-3\delta \right ) - [1-\Psi(x)] + 3\eps \\
                                          <\ & 5\eps,
  \end{align}
  where the first, second, and third inequalities follow from Steps I, II, and III respectively. Similarly, we have for any $m \geq \max\{M_1, M_2, M_3, M_4\}$, $\Pin$, and $x \in \RR$,
  \begin{equation}
    \P \left ( \supkm \left \{ \frac{S_k^2}{\widehat \sigma_k^2k} - \log \left ( \frac{k}{m} \right ) \right \} \geq x \right ) - [1-\Psi(x)] > -5\eps.
  \end{equation}
  Since $\eps$ was arbitrary, we have the desired result:
  \begin{equation}
    \lim_\mto \sup_\Pin \sup_{x \in \RR} \left \lvert \P \left ( \supkm \left \{ \frac{S_k^2}{\widehat \sigma_k^2k} - \log \left ( \frac{k}{m} \right ) \right \} \geq x \right ) - [1-\Psi(x)] \right \rvert = 0,
  \end{equation}
  which completes the proof.
\end{proof}

\begin{proof}[\proofpreamble{}for the one-sided case in \eqref{eq:main-theorem-one-sided-general}]
  The proof proceeds similarly to that of the two-sided case but with some additional arguments due to some complications arising from the presence of the function $g$. Defining $g^\star(x) = x^2 + 2 \log \Phi(x)$ for $x \in \RR$, let $a_0 := (g^\star)^{-1}(0) \approx 0.7286$, and recall that $g(x) = g^\star(x) \lor 0$.
  Moreover, let $a_1 := \Psiplus^{-1}(0.5)\approx 0.92575$. In particular, we note that $\Delta := a_1 - a_0 > 0.1$, and we will use this fact later on. 

  Let $\eps > 0$ be arbitrary. By \cref{proposition:properties of psi}, we have that $\Psiplus$ is uniformly continuous on $[a_0, \infty)$. As such, choose $\delta \in (0, \Delta/3)$ sufficiently small so that for any $y,z \geq a_0$ with $|y-z| \leq 3\delta$, it holds that
  \begin{equation}
    \left \lvert \Psiplus(y) - \Psiplus(z) \right \rvert < \eps.
  \end{equation}
  Further appealing to uniform continuity of $g(\sqrt{\cdot })$ on $[a_0, \infty)$ (see \cref{lemma:Lipschitzness of gsqrt}), choose $\gamma > 0$ so that for any $y,z \geq a_0$ such that $|y-z| \leq \gamma$, it holds that
  \begin{equation}
    |g(\sqrt{y}) - g(\sqrt{z})| < \delta.
  \end{equation}
  Take $M_1$ large enough so that
  \begin{equation}
    \sup_\Pin \P \left ( \supkm |R_k| \geq \delta \right ) \leq \eps.
  \end{equation}

  \paragraph{Step I: Handling errors incurred from estimating variances}
  Similar to Step I of the proof of the two-sided case, we have by \cref{theorem:consistent-variance-estimation} that $\widehat \sigma_n^2 / \sigma_\P^2 - 1 = \oPcalas\left (\log^{-1}(n)\right )$ and hence we take
  $M_2$ sufficiently large so that for all $m \geq M_2$,
  \begin{equation}
    \sup_\Pin \P \left ( \supkm \log (k) |\widehat \sigma_k^2 / \sigma_\P^2 - 1| \geq \delta \right ) < \eps.
  \end{equation}
  Using monotonicity of $g$ and multiplicative ``continuity'' of $g$ (see \cref{lemma:g-properties}), we have that for all $m \geq \max \{M_1, M_2\}$, all $x \geq 0$, and all $\Pin$, 
  \begin{align}
    &\P \left ( \supkm \left \{ g \left ( \frac{S_k \lor 0}{\widehat \sigma_k \sqrt{k}} \right ) + R_k - \log \left ( \frac{k}{m} \right ) \right \} \geq x \right )\\
    \leq\ &\P \left ( \exists k \geq m :  g \left ( \frac{(S_k \lor 0) / \sqrt{\sigma_\P^2 k }}{ \sqrt{\left (1 - \delta / \log(k) \right )}} \right ) - \log \left ( \frac{k}{m} \right ) \geq x - \delta \right ) + \eps + \P \left ( \supkm \log(k)|\widehat \sigma_k / \sigma_\P - 1 | \geq \delta \right ) \\
    \leq\ &\P \left ( \exists k \geq m :   \frac{g\left((S_k \lor 0) / \sqrt{\sigma_\P^2 k}\right )}{1-\delta / \log(k)} - \log \left ( \frac{k}{m} \right ) \geq x - \delta \right ) + 2\eps \\
    \leq\ &\P \left ( \exists k \geq m : g \left ( \frac{S_k \lor 0}{\sigma_\P \sqrt{k}} \right ) - \log \left ( \frac{k}{m} \right )  \geq x - 2\delta \right ) + 2\eps.
  \end{align}
  Similarly, it holds that
  \begin{align}
    \P \left ( \supkm \left \{ g \left ( \frac{S_k \lor 0}{\widehat \sigma_k \sqrt{k}} \right ) + R_k - \log \left ( \frac{k}{m} \right ) \right \} \geq x \right ) &\geq \P \left ( \exists k \geq m : g \left ( \frac{S_k \lor 0}{\sigma_\P \sqrt{k}}  \right ) - \log \left ( \frac{k}{m} \right ) \geq x + 2\delta \right ) - 2\eps.
  \end{align}

  \paragraph{Step II: Controlling the strong Gaussian approximation error}
  Using \cref{theorem:strong-gaussian-approx} \citep[Theorem 4]{waudby2025nonasymptotic}, there exist \iid{} Gaussian random variables $\infseqn{Y_n}$ with mean zero and variance $\Var_\P(X_1)$ so that once we put $G_n := \sum_{i=1}^n Y_i$ for any $n \in \NN$, 
  \begin{align}
    \forall s > 0,\quad \lim_\mto \sup_\Pin \P \left ( \supkm \left \lvert \frac{S_k - G_k}{k^{1/q}} \right \rvert \geq s \right )  = 0,
  \end{align}
  and hence
  \begin{equation}
    S_n^2 = \left ( G_n + \oPcalas(n^{1/q}) \right )^2 = G_n^2 + \oPcalas(n^{r/2})
  \end{equation}
  for some $r > 2$.
  As such, take $M_3$ sufficiently large so that for all $m \geq M_3$,
  \begin{equation}
    \sup_\Pin \P \left ( \supkm \frac{|(S_k \lor 0)^2 - (G_k \lor 0)^2|}{k} \geq \ubar \sigma^2 \gamma \right ) < \eps.
  \end{equation}
  Appealing to monotonicity and uniform continuity of $g$, as well as the definition of $\gamma$, it holds that
  \begin{align}
    &\P \left ( \supkm \left \{  g \left ( \frac{S_k \lor 0}{\sigma_\P \sqrt{k}} \right ) - \log \left ( \frac{k}{m} \right )  \right \} \geq x - 2\delta \right )\\
    \leq\ &\P \left ( \exists k \geq m :  g \left ( \sqrt{\frac{(G_k \lor 0)^2}{\sigma_\P^2 k} + \frac{\cancel{k} \ubar \sigma^2 \gamma}{\sigma_\P^2 \cancel{k}}} \right ) - \log \left ( \frac{k}{m} \right ) \geq x - 2\delta \right ) + \eps \\
    \leq\  &\P \left ( \supkm \left \{ g \left ( \frac{G_k \lor 0}{\sigma_\P \sqrt{k}} \right ) - \log \left ( \frac{k}{m} \right ) \right \} \geq x - 3\delta \right ) + \eps.
  \end{align}
  Through a similar argument, it holds that
  \begin{align}
    &\P \left ( \supkm \left \{ g \left ( \frac{S_k \lor 0}{\sigma_\P \sqrt{k}}\right ) - \log \left ( \frac{k}{m}  \right ) \right \} \geq x+2\delta \right ) \geq \P \left ( \supkm \left \{ g \left ( \frac{G_k \lor 0}{\sigma_\P \sqrt{k}} \right )- \log \left ( \frac{k}{m}  \right ) \right \} \geq x+ 3\delta \right ) - \eps.
  \end{align}

  \paragraph{Step III: Appealing to quantile-uniform convergence to the Robbins-Siegmund distribution}
  Using \cref{proposition:robbins-siegmund from Gaussians}, take $M_4$ large enough so that for all $m \geq M_4$,
 \begin{equation}
   \sup_\Pin \sup_{x \geq 0}\left\lvert \P \left ( \supkm \left \{ g \left ( \frac{G_k \lor 0}{\sigma_\P \sqrt{k}} \right ) - \log \left ( \frac{k}{m} \right ) \right \} \geq x \right ) -[1-\Psiplus(x)] \right \rvert < \eps.
 \end{equation} 
 Further appealing to uniform continuity of $\Psiplus$ (\cref{proposition:properties of psi}) and recalling that $\delta \in (0, \Delta/3)$ was constructed so that $|\Psiplus(y) - \Psiplus(z)| < \eps$ whenever $|y-z| \leq 3\delta$, we note that for any $m\geq M_4$, $\Pin$, and any $x \geq a_1$,
 \begin{align}
   &\P \left ( \supkm \left \{ g \left ( \frac{(G_k \lor 0)^2}{\sigma_\P^2 k} \right ) - \log \left ( \frac{k}{m} \right ) \right \} \geq x - 2\delta \right ) -[1-\Psiplus(x)]\\
   <\ &\eps +   [1-\Psiplus(x)] - [1-\Psiplus(x-2\delta)]  \\
   <\ &2\eps.
 \end{align}
 Similarly, we have for any $\Pin$ and any $x \geq a_1$,
 \begin{equation}
   \P \left ( \supkm \left \{ g \left ( \frac{(G_k \lor 0)^2}{\sigma_\P^2 k} \right ) - \log \left ( \frac{k}{m} \right ) \right \} \geq x + 2\delta \right ) -[1-\Psiplus(x)] > -2\eps.
 \end{equation}

 \paragraph{Step IV: Putting the previous steps together}
 It holds that for all $m \geq \max \{ M_1, M_2, M_3, M_4 \}$, all $\Pin$, and all $x \geq a_1$,
 \begin{align}
   & \P \left ( \supkm \left \{ g \left ( \frac{S_k \lor 0}{\widehat \sigma_k \sqrt{k}} \right ) + R_k - \log \left ( \frac{k}{m} \right ) \right \} \geq x \right ) -[1-\Psiplus(x)] \\
   <\ & \P \left ( \supkm \left \{ g \left ( \frac{S_k \lor 0}{\sigma_\P \sqrt{k}} \right ) - \log \left ( \frac{k}{m} \right ) \right \} \geq x - \delta \right ) -[1-\Psiplus(x)]  + \eps\\
   <\ & \P \left ( \supkm \left \{ g \left ( \frac{G_k \lor 0}{\widehat \sigma_k \sqrt{k}} \right ) - \log \left ( \frac{k}{m} \right ) \right \} \geq x - 2\delta \right ) -[1-\Psiplus(x)]  + 2\eps\\
   <\ &4\eps,
 \end{align}
 where the  first, second, and third inequalities follow from Steps I, II, and III, respectively. Similarly, it holds that
 \begin{equation}
   \P \left ( \supkm \left \{ g \left ( \frac{S_k \lor 0}{\widehat \sigma_k \sqrt{k}} \right ) - \log \left ( \frac{k}{m} \right ) \right \} \geq x \right ) -[1-\Psiplus(x)] > - 4\eps.
 \end{equation}
 Since $\eps$ was arbitrary, the desired result follows:
 \begin{equation}
   \lim_\mto \sup_\Pin \sup_{x\geq a_1} \left \lvert \P \left ( \supkm \left \{ g \left ( \frac{S_k \lor 0}{\widehat \sigma_k \sqrt{k}} \right ) - \log \left ( \frac{k}{m} \right ) \right \} \geq x \right ) -[1-\Psiplus(x)] \right \rvert = 0,
 \end{equation}
 completing the proof.
\end{proof}

\subsection{Proof of \cref{theorem:two-sided inference}}\label{proof:two-sided inference}
\begin{proof}[\proofpreamble{}of \cref{theorem:two-sided inference}]
  Starting with the claim that $\infseqm{\infseqkm{\widebar p_k^\brackm}}$ is a $\Pcal_0$-uniform anytime $p$-value, fix $\Pnullin$, $m \in \NN$, and $\alpha \in (0, 1)$. We have
  \begin{align}
    \P \left ( \exists k \geq m : \widebar p_k^\brackm \leq \alpha \right ) &= \P \left ( \exists k \geq m :  \Psi(k\widehat \mu_k^2 / \widehat \sigma_k^2 - \log(k/m)) \geq 1- \alpha \right )\\
                                                                            &= \P \left ( \supkm \left \{  k\widehat \mu_k^2 / \widehat \sigma_k^2 - \log(k/m) \right \} \geq \Psi^{-1}(1- \alpha) \right ).
  \end{align}
  Using the fact that $\Psi$ is invertible between $[0, \infty)$ and $[0,1)$, we have 
  \begin{align}
    \sup_{\alpha \in (0, 1)}\left \lvert \P \left ( \exists k \geq m : \widebar p_k^\brackm \leq \alpha \right ) - \alpha \right \rvert &= \sup_{\alpha \in (0, 1)}\left \lvert \P \left ( \supkm \left \{  k\widehat \mu_k^2 / \widehat \sigma_k^2 - \log(k/m) \right \} \geq \Psi^{-1}(1- \alpha) \right ) - \alpha \right \rvert\\
    &= \sup_{x \geq 0}\left \lvert \P \left ( \supkm \left \{  k\widehat \mu_k^2 / \widehat \sigma_k^2 - \log(k/m) \right \} \geq x \right ) - [1-\Psi(x)] \right \rvert.
  \end{align}
  Applying \cref{theorem:main-convergence-in-distribution} yields
  \begin{equation}
    \lim_\mto \sup_\Pnullin \sup_{\alpha \in (0, 1)} \left \lvert \P \left ( \exists k \geq m : \widebar p_k^\brackm \leq \alpha \right ) - \alpha \right \rvert= 0,
  \end{equation}
  completing the proof for the first claim.

  The claim that $\infseqm{\infseqkm{\widebar C_k^\brackm}}$ is a $\Pcal$-uniform $(1-\alpha)$-\cs{} follows similarly. Indeed, observe that for any $\Pin$, $m \in \NN$, and any $\alpha \in (0, 1)$, the miscoverage event can be re-written as
  \begin{align}
    &\left \{ \omega \in \Omega : \exists k \geq m \text{ such that } \EE_\P[X_1] \notin \left [ \widehat \mu_k \pm \widehat \sigma_k \sqrt{[\Psi^{-1} (1-\alpha) + \log (k /m )] /k} \right ] \right \}\\
    =\ & \left \{ \omega \in \Omega : \exists k \geq m \text{ such that } \left \lvert \widehat \mu_k - \EE_\P[X_1] \right \rvert \geq \widehat \sigma_k \sqrt{[\Psi^{-1} (1-\alpha) + \log (k / m) ] / k} \right \}\\
    =\ &\left \{ \omega \in \Omega : \supkm \left \{ \frac{k (\widehat \mu_k - \EE_\P[X_1])^2}{\widehat \sigma_k^2} - \log( k / m ) \right \} \geq \Psi^{-1}(1-\alpha) \right \}.
  \end{align}
  The result follows from the same line of reasoning as in the case of the anytime $p$-value. This completes the proof of \cref{theorem:two-sided inference}.
\end{proof}

\subsection{Proof of \cref{theorem:one-sided inference}}\label{proof:one-sided inference}

\begin{proof}[\proofpreamble{}of \cref{theorem:one-sided inference}]
  Fix $\P \in \Pcal_0^\leq$, $m \in \NN$, and $\alpha \in (0, 1)$. We have
  \begin{align}
    \P \left ( \exists k \geq m : \widebar p_k^\brackm \leq \alpha \right ) &= \P \left ( \exists k \geq m :  \Psiplus \left ( g \left ( \sqrt{k(\widehat \mu_k\lor 0)^2} / \widehat \sigma_k \right ) - \log(k/m) \right ) \geq 1- \alpha \right )\\
                                                                            &= \P \left ( \exists k \geq m :  g \left ( \sqrt{k(\widehat \mu_k\lor 0)^2} / \widehat \sigma_k \right ) - \log(k/m) \geq \Psiplus^{-1}(1- \alpha) \right ).
  \end{align}
  Using the fact that $\Psiplus$ is invertible on the interval $[\Psiplus(0),1)$, we have
  \begin{align}
    \sup_{\alpha \in (0, 1/2]}\left \lvert \P \left ( \exists k \geq m : \widebar p_k^\brackm \leq \alpha \right ) - \alpha \right \rvert &= \sup_{\alpha \in (0, 1/2]}\left \lvert \P \left ( \supkm \left \{  g \left ( \frac{\sqrt{k} (\widehat \mu_k \lor 0)}{\widehat \sigma_k} \right ) - \log \left ( \frac{k}{m} \right ) \right \} \geq \Psi^{-1}(1- \alpha) \right ) - \alpha \right \rvert\\
&= \sup_{x \geq a_1}\left \lvert \P \left ( \supkm \left \{  g \left ( \frac{\sqrt{k} (\widehat \mu_k \lor 0)}{\widehat \sigma_k} \right ) - \log \left ( \frac{k}{m} \right ) \right \} \geq x \right ) - [1-\Psiplus(x)] \right \rvert,
  \end{align}
  where we recall that $a_1 = \Psiplus^{-1}(0.5)$.
  Applying \cref{theorem:main-convergence-in-distribution} yields the desired result:
  \begin{equation}
    \lim_\mto \sup_{\P \in \Pcal_0^\leq}\sup_{\alpha \in (0, 1/2]}\left \lvert \P \left ( \exists k \geq m : \widebar p_k^\brackm \leq \alpha \right ) - \alpha \right \rvert = 0.
  \end{equation}
  which completes the proof.
\end{proof}

\subsection{Proof of \cref{theorem:hardness}}\label{proof:hardness}

\begin{proof}
  Suppose for the sake of contradiction that there exists a potentially randomized test $(\widebar \Gamma_k^\brackm)_{k = m}^\infty$ so that for some $\alpha \in (0, 1)$, we have both
  \begin{equation}\label{eq:hardness-type-i-err}
    \limsup_{m \to \infty} \sup_{\P \in \Pcal_0^\star} \P \left ( \exists k \geq m : \widebar \Gamma_k^\brackm = 1 \right ) \leq \alpha
  \end{equation}
  and
  \begin{equation}
    \sup_{\P \in \Pcal_1^\star} \limsup_{m \to \infty} \P \left ( \exists k \geq m : \widebar \Gamma_k^\brackm = 1 \right ) > \alpha.
  \end{equation}
  Then there must exist $\eps > 0$ so that we can always find $m_1$ arbitrarily large and nevertheless satisfy
  \begin{equation}\label{eq:hardness-power}
    \sup_{\P \in \Pcal_1^\star}\P \left ( \exists k \geq m_1 : \widebar \Gamma_k^{(m_1)} = 1 \right ) > \alpha + \eps.
  \end{equation}
  Furthermore, by \eqref{eq:hardness-type-i-err}, there exists $m_0 \geq 1$ large enough so that for all $m \geq m_0$,
  \begin{equation}
   \sup_{\P \in \Pcal_0^\star} \P \left ( \exists k \geq m : \widebar \Gamma_k^{(m)} = 1 \right ) < \alpha + \eps.
  \end{equation}
  In particular, choose some $m_1 \geq m_0$ so that \eqref{eq:hardness-power} holds.
  Notice that the events
  \begin{equation}
    A_M := \{ \widebar \Gamma_k^{(m_1)} = 1 \text{ for some } m_1 \leq k \leq M\}
  \end{equation}
  are nested for $M = m_1, m_1+1, \dots$ and that $A_M \to A := \{ \exists k \geq m_1 : \widebar \Gamma_k^{(m_1)} = 1\}$ as $M \to \infty$. Consequently, there must exist some $M^\star$ such that
 \begin{equation}\label{eq:hardness-power-finitetime}
    \sup_{\P \in \Pcal_1^\star}\P \left ( \max_{m_1 \leq k \leq M^\star}\widebar \Gamma_k^{(m_1)} = 1 \right ) > \alpha + \eps.
  \end{equation}
  On the other hand, notice that by virtue of being a $\Pcal_0^\star$-uniform anytime valid test and the fact that $m_1 \geq m_0$, we have that $\max_{m_1 \leq k \leq M^\star}\widebar \Gamma_k^{(m_1)}$ uniformly controls the type-I error under $\Pcal_0^\star$, i.e.
  \begin{equation}
    \sup_{\P \in \Pcal_0^\star} \P \left ( \max_{m_1 \leq k \leq M^\star} \widebar \Gamma_k^{(m_1)} = 1 \right ) \leq \sup_{\P \in \Pcal_0^\star} \P \left ( \exists k \geq m_1 : \widebar \Gamma_k^{(m_1)} = 1 \right ) < \alpha + \eps.
  \end{equation}
  Combining the above with the hardness result of \citet[Theorem 2]{shah2020hardness} applied to the test $\max_{m_1 \leq k \leq M^\star} \widebar \Gamma_k^{(m_1)}$, we have that
  \begin{equation}
    \sup_{\P \in \Pcal_1^\star} \P \left ( \max_{m_1 \leq k \leq M^\star} \widebar \Gamma_k^{(m_1)} = 1 \right ) < \alpha + \eps,
  \end{equation}
  contradicting \eqref{eq:hardness-power-finitetime}, and thus completing the proof of \cref{theorem:hardness}.
\end{proof}

\subsection{Proof of \cref{theorem:seq-GCM}}\label{proof:seq-GCM}

  Before proceeding with the proof, notice that the estimated residual $R_i$ can be written as
  \begin{equation}\label{eq:seqgcm-decomposition}
    R_i = \xi_i + b_{i} + \nu_i
  \end{equation}
  where $\xi_i:= \xi_i^x \cdot \xi_i^y$ is a true product residual with
  \begin{align}
    \xi_i^x := \{X_i - \mu^x(Z_i) \} ~~~\text{and}~~~ \xi_i^y := \{ Y_i - \mu^y (Z_i) \},
  \end{align}
  $b_i$ is a product regression error term given by
  \begin{equation}
    b_{i} := \left \{\widehat \mu_i^x(Z_i) - \mu^x(Z_i) \right \} \left \{ \widehat \mu_i^y(Z_i) - \mu^y(Z_i) \right \},
  \end{equation}
  and $\nu_i := \nu_i^\xy + \nu_i^\yx$ is a cross-term where
  \begin{align}
    \nu_{i}^{x,y}&:= \left \{ \widehat \mu_i^x(Z_i) - \mu^x(Z_i) \right \} \xi_i^y,~~\text{and}\\
    \nu_{i}^{y,x}&:= \left \{ \widehat \mu_i^y(Z_i) - \mu^y(Z_i) \right \} \xi_i^x.
  \end{align}
              Furthermore, define their averages as $\widebar b_n := \frac{1}{n} \sum_{i=1}^n b_i$ and similarly for $\widebar \nu_n^\xy$, $\widebar \nu_n^\yx$, and $\widebar \xi_n$. We may at times omit the argument $(Z_i)$ from $\widehat \mu_{i}^x(Z_i) \equiv \widehat \mu_i^x$ or $\mu^x(Z_i) \equiv \mu^x$ etc.~when it is clear from context.
  With these shorthands in mind, we are ready to prove \cref{theorem:seq-GCM}.
\begin{proof}[\proofpreamble{}of \cref{theorem:seq-GCM}]
  Note that it suffices to show that $\sup_{k \geq m} \left \{ k\GCMWS_k^2 - \log (k / m) \right \}$ converges $\Pnull$-uniformly to the two-sided Robbins-Siegmund distribution as $\mto$.
  Begin by writing $\GCMWS_n$ as
  \begin{align}
    \GCMWS_n &:= \frac{1}{n \widehat \sigma_n} \sum_{i=1}^n R_i\\
    &\equiv \widehat \sigma_n^{-1} \left (\widebar \xi_n + \widebar \nu_n + \widebar b_n \right )
  \end{align}
  and through a direct calculation, notice that our squared GCM statistic can be written as
  \begin{align}
    \GCMWS_n^2 &= \frac{\widebar \xi_n^2 + 2\widebar \xi_n (\widebar \nu_n + \widebar b_n) + (\widebar \nu_n + \widebar b_n)^2}{\widehat \sigma_n^2}\\
    &= \underbrace{\frac{\widebar \xi_n^2}{\widehat \sigma_n^2}}_{(i)} + \underbrace{\frac{2\widebar \xi_n (\widebar \nu_n + \widebar b_n)}{\widehat \sigma_n^2}}_{(ii)} + \underbrace{\frac{(\widebar \nu_n + \widebar b_n)^2}{\widehat \sigma_n^2}}_{(iii)}.
  \end{align}
  In the discussion to follow, we analyze these three terms separately (in Steps 1, 2, and 3, respectively) and combine them to yield the desired result in Step 4.

  \paragraph{Step I: Analyzing the denominator of $(i)$}
  In \cref{lemma:consistent-variance-estimation}, we show that under the assumptions of \cref{theorem:seq-GCM}, the estimator $\widehat \sigma_n^2 := \frac{1}{n} \sum_{i=1}^n R_i^2$ is $\Pnull$-uniformly consistent for $\Var(\xi) \equiv \EE (\xi^2)$ on a multiplicative scale at a rate faster than $1/\log n$, meaning
  \begin{equation}
    \frac{\widehat \sigma_n^2}{\Var(\xi)} -1 = \oPnullas \left ( \frac{1}{\log n}\right ).
  \end{equation}
  We will later use the fact that the denominator of $(i)$ satisfies the above.

  \paragraph{Step II: Analyzing $(ii)$} In Lemmas~\ref{lemma:bias-convergence} and~\ref{lemma:nu-convergence}, we show that under the assumptions of \cref{theorem:seq-GCM}, $\widebar b_n = \oPnullas(1 / \sqrt{n \log \log n})$ and $\widebar \nu_n = \oPnullas(1/\sqrt{n \log \log n})$, respectively. By the uniform law of the iterated logarithm of \citet[Corollary 3.3]{ruf2025concentration}, it holds that
  \begin{equation}
    \widebar \xi_n = \OPnullas(\sqrt{\log \log n / n}).
  \end{equation}
  Combining these four convergence results together with the calculus outlined in \cref{proposition:oPcalas-calculus}, we have
  \begin{align}
    (ii) &= \frac{2\widebar \xi_n \cdot (\widebar \nu_n + \widebar b_n)}{\widehat \sigma_n^2} \\
    &= \frac{\OPnullas(\sqrt{\log \log n / n}) \cdot \oPnullas(1/\sqrt{n\log \log n})}{\ubar \sigma^2 \cdot (1 + \oPnullas(1))}\\
    &= \oPnullas \left ( \frac{1}{n} \right ).
  \end{align}

  \paragraph{Step III: Analyzing $(iii)$} Again by Lemmas~\ref{lemma:bias-convergence} and~\ref{lemma:nu-convergence} and the calculus of \cref{proposition:oPcalas-calculus}, we have that
  \begin{align}
    (iii) &\leq \left \lvert \frac{(\widebar \nu_n + \widebar b_n)^2}{\widehat \sigma_n^2} \right \rvert\\
          &= \left \lvert \frac{\oPnullas(1/n)}{\ubar \sigma^2 \cdot (1 + \oPnullas(1))} \right \rvert \\
    &= \oPnullas \left ( \frac{1}{n} \right ).
  \end{align}

  \paragraph{Step IV: Putting $(i)$--$(iii)$ together} Recalling that $\GCMWS^2_n = (i) + (ii) + (iii)$ and noting that $(ii) + (iii) = \oPcalas(1/n)$, we have
  \begin{equation}
    n\GCMWS_n^2 = \frac{\left ( \sum_{i=1}^n \xi_i \right )^2}{n \widehat \sigma_n^2} + \oPcalas(1).
  \end{equation}
  Applying \cref{theorem:generalized convergence,lemma:consistent-variance-estimation}, we have
  \begin{equation}
    \lim_\mto \sup_\Pnullin \sup_{x \geq 0} \left \lvert \P \left ( \supkm \left \{ k\GCMWS_k^2 - \log \left ( \frac{k}{m} \right ) \right \} \geq x \right ) - [1-\Psi(x)] \right \rvert = 0.
  \end{equation}
  Using a similar rearrangement of terms to what was found in the proof of \cref{theorem:two-sided inference}, we have that
  \begin{equation}
    \lim_\mto \sup_\Pnullin \sup_{\alpha \in (0, 1)} \left \lvert \P \left ( \exists k \geq m : 1-\Psi \left ( k \GCMWS_k^2 - \log \left ( \frac{k}{m} \right ) \right )  \leq \alpha\right ) - \alpha \right \rvert = 0,
  \end{equation}
  which completes the proof of \cref{theorem:seq-GCM}.
\end{proof}

\begin{lemma}[$\Pnull$-uniformly strongly consistent variance estimation]\label{lemma:consistent-variance-estimation}
  Let $\widehat \sigma_n^2$ be the sample variance of $R_i$:
  \begin{equation}
    \widehat \sigma_n^2 := \frac{1}{n}\sum_{i=1}^n R_{i}^2 - \left ( \frac{1}{n} \sum_{i=1}^n R_i \right )^2.
  \end{equation}
  Then, 
  \begin{equation}
    \frac{\widehat \sigma_n^2}{\EE(\xi^2)} - 1 = \oPnullas \left ( \frac{1}{\log n} \right ).
  \end{equation}
\end{lemma}

\begin{proof}[\proofpreamble{}of \cref{lemma:consistent-variance-estimation}]
  First, consider the following decomposition:
  \begin{align}
    R_{i}^2 &= \left [ \xi_i^x \xi_i^y + \xi_i^x \left \{ \mu^y-\widehat \mu^y_{i} \right \} + \xi_i^y \left \{ \mu^x - \widehat \mu^x_{i} \right \} + \left ( \widehat \mu^x_{i} - \mu^x \right ) \left ( \widehat \mu^y_{i} -  \mu^y \right ) \right ]^2 \\
              &= \xi_i^2 +\\
              &\quad \underbrace{2 (\xi_i^x)^2 \xi_i^y \left \{ \mu^y - \widehat \mu^y_{i} \right \} + 2 (\xi_i^y)^2 \xi_i^x \left \{ \widehat \mu^x_{i} - \mu^x \right \}}_{\I_i} + \\
    &\quad \underbrace{4 \xi_i \left \{ \mu^x - \widehat \mu^x_{i} \right \} \left \{ \mu^y - \widehat \mu^y_{i} \right \}}_{\II_i} +\\
              &\quad \underbrace{2 \xi_i^x \left \{ \mu^y - \widehat \mu^y_{i} \right \}^2 \left \{ \mu^x - \widehat \mu^x_{i} \right \} + 2 \xi_i^y \left \{ \mu^x - \widehat \mu^x_{i} \right \}^2 \left \{ \mu^y - \widehat \mu^y_{i} \right \}}_{\III_i}+\\
                &\quad \underbrace{\left \{ \mu^x - \widehat \mu^x_{i} \right \}^2\left \{ \mu^y - \widehat \mu^y_{i} \right \}^2}_{\IV_i}.
  \end{align}
  Letting $\widebar \I_n := \frac{1}{n} \sum_{i=1}^n  \I_i$ and similarly for $\widebar \II_n$, $\widebar \III_n$, and $\widebar \IV_n$, we have that
  \begin{equation}
    \widehat \sigma_n^2 = \frac{1}{n} \sum_{i=1}^n  \xi_i^2 + \widebar \I_n + \widebar \II_n + \widebar \III_n + \widebar \IV_n - (\widebar R_n)^2
  \end{equation}
  and we will separately show that $\widebar \I_n$, $\widebar \II_n$, $\widebar \III_n$, $\widebar \IV_n$, and $(\widebar R_n)^2$ are all $\oPnullas(1 / \log n)$.

  \paragraph{Step I: Convergence of $\widebar \I_n$}
  By the Cauchy-Schwarz inequality, we have that
  \begin{align}
    \frac{1}{n}\sum_{i=1}^n (\xi_i^x)^2 \xi_i^y \left \{ \mu^y - \widehat \mu^y_{i} \right \} &\leq  \underbrace{\sqrt{ \frac{1}{n}\sum_{i=1}^n (\xi_i^x\xi_i^y)^2 }}_{(\star)}\cdot \underbrace{\sqrt{ \frac{1}{n}\sum_{i=1}^n(\xi_i^x)^2 \left \{ \mu^y - \widehat \mu^y_{i} \right \}^2  }}_{(\dagger)}.
  \end{align}
  Now, writing $\xi_i := \xi_i^x \xi_i^y$, notice that
  \begin{align}
    (\star) &\equiv \frac{1}{n} \sum_{i=1}^n \xi_i^2\\
            &\leq \left (\frac{1}{n} \sum_{i=1}^n \xi_i^2 - \EE(\xi_i^2) \right ) + \EE(\xi_i^2) \\
    &= \oPnullas(1) + \EE \left [ \left ( |\xi_i|^{2+\delta} \right )^{\frac{2}{2+\delta}} \right ]\\
            &\leq \oPnullas(1) + \left ( \EE |\xi_i|^{2+\delta} \right )^{\frac{2}{2+\delta}}\\
    &\leq \OPnullas(1),
  \end{align}
  where the last line follows from \cref{assumption:GCM-finite-moments}. Moreover, by \cref{lemma:convergence-of-cross-terms}, we have that $(\dagger) = \oPnullas(1/ \log n)$, and hence by \cref{proposition:oPcalas-calculus}, $\widebar \I_n \leq (\star)\cdot (\dagger) = \oPnullas(1/\log n)$.
  
  \paragraph{Step II: Convergence of $\widebar \II_n$} Again by Cauchy-Schwarz, we have
  \begin{align}
    \frac{1}{n} \sum_{i=1}^n \xi^x_i \xi_i^y \{ \mu^x - \widehat \mu_{i}^x \} \{ \mu^y - \widehat \mu_{i}^y \} \leq \sqrt{\frac{1}{n} \sum_{i=1}^n (\xi_i^x)^2 \{ \mu^y - \widehat \mu_{i}^y\}^2 } \cdot \sqrt{\frac{1}{n} \sum_{i=1}^n (\xi_i^y)^2 \{ \mu^x - \widehat \mu_{i}^x\}^2},
  \end{align}
  and hence again by \cref{lemma:convergence-of-cross-terms}, we have $\widebar \II_n = \oPnullas(1 /\log n)$.

  \paragraph{Step III: Convergence of $\widebar \III_n$}
  Following \citet[Section D.1]{shah2020hardness} and using the inequality $2|ab| \leq a^2 + b^2$ for any $a, b \in \RR$, we have 
  \begin{align}
    &\frac{2}{n} \sum_{i=1}^n \xi_i^x \{\mu^y - \widehat \mu_{i}^y \}^2 \{ \mu^x - \widehat \mu_{i}^x \} \\
    \leq\ &\frac{1}{n} \sum_{i=1}^n (\xi_i^x)^2 \{ \mu^y - \widehat \mu_{i}^y \}^2 + \frac{1}{n} \sum_{i=1}^n \{ \mu^y - \widehat \mu_{i}^y \}^2 \{ \mu^x - \widehat \mu_{i}^x \}^2,
  \end{align}
  and hence by Lemmas~\ref{lemma:convergence-of-cross-terms} and \ref{lemma:bias-convergence}, we have $\widebar \III_n = \oPnullas(1/ \log n)$.

  \paragraph{Step IV: Convergence of $\widebar \IV_n$}
  First, notice that
  \begin{align}
    \widebar \IV_n&:= \frac{1}{n} \sum_{i=1}^n \left \{ \mu^x - \widehat \mu_{i}^x \right \}^2 \cdot \left \{ \mu^y - \widehat \mu_{i}^y \right \}^2\\
    \ifverbose %
    &\leq \frac{1}{n} \sum_{i=1}^n \left \{ \mu^x - \widehat \mu_{i}^x \right \}^2 \cdot \sum_{i=1}^n \left \{ \mu^y - \widehat \mu_{i}^y \right \}^2\\
    \fi %
    &\leq n \cdot \frac{1}{n} \sum_{i=1}^n \left \{ \mu^x - \widehat \mu_{i}^x \right \}^2 \cdot \frac{1}{n}\sum_{i=1}^n \left \{ \mu^y - \widehat \mu_{i}^y \right \}^2.
  \end{align}
  Applying Lemmas~\ref{lemma:bias-convergence} and \ref{proposition:oPcalas-calculus}, we have that $\widebar \IV_n = \oPnullas(1/ \log n)$.

  \paragraph{Step V: Convergence of $(\widebar R_n)^2$ to 0}
  We will show that $(\widebar R_n)^2 = \oPnullas(1/\log n)$.
  Using the decomposition in \eqref{eq:seqgcm-decomposition} at the outset of the proof of \cref{theorem:seq-GCM}, we have that
  \begin{equation}
    \widebar R_n := \widebar \xi_n + \widebar b_n + \widebar \nu_n.
  \end{equation}
  Therefore, we can write its square as
  \begin{equation}
    (\widebar R_n)^2 = (\widebar \xi_n)^2 + 2\widebar \xi_n \cdot \left ( \widebar b_n + \widebar \nu_n \right ) + (\widebar b_n + \widebar \nu_n)^2.
  \end{equation}
  By \cref{assumption:GCM-finite-moments}, we have that there exists a $\delta > 0$ so that $\sup_\Pnullin \EE_\P|\xi|^{2+\delta} < \infty$. By the de la Vall\'ee-Poussin criterion for uniform integrability, we have that the $(1+\delta)^\text{th}$ moment of $\xi$ is uniformly integrable:
  \begin{equation}
    \lim_\mto \sup_\Pnullin \EE_\P \left ( |\xi|^{1+\delta} \1\{ |\xi|^{1+\delta} \geq m \} \right ) = 0.
  \end{equation}
  By \citet[Theorem 1]{waudby2024distribution}, we have that $\widebar \xi_n = \oPnullas \left ( n^{1/(1+\delta) - 1} \right )$, and in particular,
  \begin{equation}
    \widebar \xi_n = \oPnullas \left ( 1/\sqrt{\log n} \right ).
  \end{equation}
  Using \cref{proposition:oPcalas-calculus}, we observe that
  \begin{equation}
    (\widebar \xi_n)^2 = \oPnullas \left ( 1/\log n \right ),
  \end{equation}
  and hence it now suffices to show that $\widebar b_n + \widebar \nu_n = \oPnullas(1/\log n)$.
  Indeed, by Lemmas~\ref{lemma:bias-convergence} and \ref{lemma:nu-convergence}, we have that $\widebar b_n = \oPnullas(1/\sqrt{n \log \log n})$ and $\widebar \nu_n = \oPnullas(1/\sqrt{n\log \log n})$, respectively. Putting these together, we have
  \begin{equation}
    (\widebar R_n)^2 = (\widebar \xi_n)^2 + 2\widebar \xi_n \cdot \left ( \widebar b_n + \widebar \nu_n \right ) + (\widebar b_n + \widebar \nu_n)^2 = \oPnullas(1/\log n),
  \end{equation}
  completing the argument for Step V.

  \paragraph{Step VI: Convergence of $\widehat \sigma_n^2$ to $\EE(\xi^2)$} Putting Steps I--V together, notice that
  \begin{align}
    \widehat \sigma_n^2 - \EE(\xi^2) &= \frac{1}{n} \sum_{i=1}^n \xi_i^2- \EE(\xi^2) + \widebar \I_n + \widebar \II_n + \widebar \III_n + \widebar \IV_n - (\widebar R_n)^2\\
    &= \frac{1}{n} \sum_{i=1}^n \xi_i^2 - \EE(\xi^2) + \oPnullas(1/ \log n).
  \end{align}
  Now, since $\sup_\Pin \EE_\P |\xi^2|^{1+\delta/2} < \infty$, $\xi^2$ we have by the de la Vall\'ee criterion for uniform integrability that $\xi^2$ has a $\Pnull$-uniformly integrable $(1+\delta/4)^\tth$ moment meaning that
  \begin{equation}
    \lim_{m \to \infty}\sup_\Pnullin \EE_\P \left [ (\xi^2)^{1+\delta/4} \1 \left \{ (\xi^2)^{1+\delta/4} > m \right \} \right ] = 0,
  \end{equation}
  and hence by \citet[Theorem 1]{waudby2024distribution}, we have that
 \begin{equation}\label{eq:variance estimation used wsetal}
   \frac{1}{n} \sum_{i=1}^n \xi_i^2 - \EE(\xi^2) = \oPcalas \left ( n^{1/(1+\delta/4) - 1} \right ),
 \end{equation} 
 and in particular, $\frac{1}{n} \sum_{i=1}^n \xi_i^2 - \EE(\xi^2) = \oPcalas(1/\log n )$, so that $\widehat \sigma_n^2 - \EE(\xi^2) = \oPcalas(1 / \log n)$. To complete the proof, we use the fact that $\ubar \sigma^2 := \inf_\Pin \EE_\P(\xi^2)$ is positive. Indeed, we have that for any $\Pin$, any $\eps > 0$, and any $m \in \NN$,
 \begin{align}
   \P \left ( \supkm \log(k) \left \lvert \frac{\widehat \sigma_k^2}{\EE_\P(\xi^2)} - 1 \right \rvert \geq \eps \right ) &= \P \left ( \supkm \frac{\log(k)}{\EE_\P(\xi^2)} \left \lvert \widehat \sigma_k^2 - \EE_\P(\xi^2) \right \rvert \geq \eps \right ) \\
                                                                                                                 &\leq \P \left ( \supkm \frac{|\widehat \sigma_k^2 - \EE_\P(\xi^2)|}{\log^{-1}(k)} \geq \eps \ubar \sigma^2 \right ).
 \end{align}
 Taking suprema over $\Pcal$ and limits as $\mto$ and relying on the remark following \eqref{eq:variance estimation used wsetal} completes the proof.
\end{proof}

\begin{lemma}[Convergence of the average bias term]\label{lemma:bias-convergence}
  Under \cref{assumption:SeqGCM-product-errors}, we have that
  \begin{equation}
    \widebar b_n \equiv \frac{1}{n} \sum_{i=1}^n b_i = \oPcalas \left ( 1/\sqrt{n \log \log n} \right ).
  \end{equation}
\end{lemma}
\begin{proof}
  Under \cref{assumption:SeqGCM-product-errors}, we have that
  \begin{equation}
    \sup_\Pnullin \EE_\P \| \widehat \mu_n^x - \mu^x \|_\LP^2 \cdot \EE_\P \| \widehat \mu_n^y - \mu^y \|_\LP^2 = O \left ( \frac{1}{n \log^{2+\delta}(n)} \right ),
  \end{equation}
  and hence let $C_\Pnull > 0$ be a constant depending only on $\Pnull$ so that
  \begin{equation}
    \sup_\Pnullin \EE_\P \| \widehat \mu_n^x - \mu^x \|_\LP^2 \cdot \EE_\P\| \widehat \mu_n^y - \mu^y \|_\LP^2 \leq \frac{C_\Pnull}{(n + 1) \log^{2+\delta/2}(n + 1) \log \log (n + 1)}
  \end{equation}
  for all $n$ sufficiently large. Consider the following series for all $k \geq m$ for any $m \geq 3$ and apply Cauchy-Schwarz to the norm $\EE_\P|\cdot |$ to obtain
  \begin{align}
    &\sup_\Pnullin \sum_{k=m}^\infty \frac{\EE_\P \left \lvert  \left \{ \widehat \mu_{k-1}^x(Z_k) - \mu^x(Z_k) \right \} \cdot \left \{ \widehat \mu_{k-1}^y(Z_k) - \mu^y(Z_k) \right \} \right \rvert}{\sqrt{k / \log \log k}}\\
    \leq\ &\sup_\Pnullin \sum_{k=m}^\infty \frac{\sqrt{\EE_\P \left \| \widehat \mu_{k-1}^x - \mu^x \right \|_\LP^2}  \cdot \sqrt{\EE_\P \left \| \widehat \mu_{k-1}^y - \mu^y \right \|_\LP^2} }{\sqrt{k / \log \log k}}\\
    =\ &\sum_{k=m}^\infty \frac{ C_\Pnull }{\sqrt{k \log^{2+\delta/2}(k) \cancel{\log \log k}} \cdot \sqrt{k / \cancel{\log \log k}}}\\
    =\ &\sum_{k=m}^\infty \frac{ C_\Pnull }{k \log^{1+\delta/4}(k) },
  \end{align}
  and since $(k \log^{1+\delta/4}(k))^{-1}$ is summable for any $\delta > 0$, we have that the above vanishes as $m \to \infty$, hence
  \begin{equation}
    \lim_\mto \sup_\Pnullin \sum_{k=m}^\infty \frac{\EE_\P \left \lvert  \left \{ \widehat \mu_{k-1}^x(Z_k) - \mu^x(Z_k) \right \} \cdot \left \{ \widehat \mu_{k-1}^y(Z_k) - \mu^y(Z_k) \right \} \right \rvert}{\sqrt{k / \log \log k}}= 0.
  \end{equation}
  Applying \citet[Theorem 2]{waudby2024distribution}, we have that
  \begin{equation}
    \widebar b_n \equiv \frac{1}{n} \sum_{k=1}^n \left \{ \widehat \mu_{k-1}^x(Z_k) - \mu^x(Z_k) \right \} \cdot \left \{ \widehat \mu_{k-1}^y(Z_k) - \mu^y(Z_k) \right \} = \oPnullas \left ( \frac{1}{\sqrt{n \log \log n}} \right ),
  \end{equation}
  which completes the proof.
\end{proof}

\begin{lemma}[Convergence of average cross-terms]\label{lemma:nu-convergence}
  Suppose that for some $\delta > 0$, and some independent $Z$ with the same distribution as $Z_n$,
  \begin{equation}\label{eq:nu-convergence-condition}
    \sup_\Pnullin \EE_\P \left[ \left ( \{ \widehat \mu_n^x(Z_n) - \mu^x(Z_n) \} \xi_n^y \right )^2 \right ] = O \left ( \frac{1}{(\log n)^{2+\delta}} \right ).
  \end{equation}
  Then,
  \begin{equation}
    \frac{1}{n} \sum_{i=1}^n \nu_i^\xy = \oPnullas(1/\sqrt{n \log \log n}),
  \end{equation}
  with an analogous statement holding when $x$ and $y$ are swapped in the above condition and conclusion.
\end{lemma}
\begin{proof}
  We will only prove the result for $\nu_i^\xy$ but the same argument goes through for $\nu_i^\yx$. Appealing to \eqref{eq:nu-convergence-condition}, let $C_\Pnull$ be a constant so that
  \begin{equation}
    \sup_\Pin \EE_\P \left[ \left ( \{ \widehat \mu_n^x(Z_n) - \mu^x(Z_n) \} \xi_n^y \right )^2 \right ] \leq \frac{C_\Pnull }{(\log n)^{2+\delta}}.
  \end{equation}
  Then notice that for all $m$ sufficiently large
 \begin{align}
   &\sup_\Pin \sum_{k=m}^\infty \frac{\EE_\P \left[ \left ( \{ \widehat \mu_k^x(Z_k) - \mu^x(Z_k) \} \xi_k^y \right )^2 \right ]}{k / \log \log k}\\
   \leq\ &\sup_\Pin \sum_{k=m}^\infty \frac{C_\Pnull}{k (\log k)^{2+\delta} / \log \log k}\\
   \leq\ &\sup_\Pin \sum_{k=m}^\infty \frac{C_\Pnull}{k (\log k)^{1+\delta}}\\
   =\ &0,
 \end{align}  
 and hence
 \begin{equation}
   \lim_\mto \sup_\Pin \sum_{k=m}^\infty \frac{\EE_\P \left[ \left ( \{ \widehat \mu_{k-1}^x(Z_k) - \mu^x(Z_k) \} \xi_k^y \right )^2 \right ]}{k / \log \log k} = 0.
 \end{equation}
 By \citet[Theorem 2]{waudby2024distribution}, we have that
 \begin{equation}
   \frac{1}{n} \sum_{i=1}^n \nu_i^\xy = \oPcalas(1/\sqrt{n \log \log n}),
 \end{equation}
 completing the proof.
\end{proof}

\begin{lemma}[Convergence of average squared cross-terms]\label{lemma:convergence-of-cross-terms}
  Under \cref{assumption:SeqGCM-regularity}, we have that
  \begin{equation}
    \frac{1}{n} \sum_{i=1}^n (\nu^\xy_i)^2 \equiv \frac{1}{n}\sum_{i=1}^n (\xi_i^x)^2 \{ \mu^y - \widehat \mu_i^y \}^2 = \oPnullas \left ( 1/\log n \right ).
  \end{equation}
  An analogous statement holds with $\xi_n^x$ replaced by $\xi_n^y$ and $\{\mu^y(Z_n) - \widehat \mu_n^y(Z_n) \}$ replaced by $\{\mu^x(Z_n) - \widehat \mu_n^x(Z_n) \}$.
\end{lemma}
\begin{proof}
  Using \cref{assumption:SeqGCM-regularity}, let $C_\Pnull > 0$ be a constant so that
  \begin{equation}
    \sup_\Pin \EE_\P \left [ (\xi_n^x)^2 \{\mu^y(Z_n) - \widehat \mu_{n-1}^y(Z_n) \}^2 \right ] \leq \frac{C_\Pnull}{ (\log n )^{2+\delta}}.
  \end{equation}
  Therefore, we have that
  \begin{align}
    &\lim_\mto \sup_\Pnullin \sum_{k=m}^\infty \frac{\EE_\P \left [ (\xi_k^x)^2 \{\mu^y(Z_k) - \widehat \mu_{k-1}^y(Z_k) \}^2 \right ]}{k (\log k)^{-1}}\\
    \leq\ & \lim_\mto \sum_{k=m}^\infty \frac{C_\Pnull }{k (\log k )^{2+\delta-1}}\\
    =\ & \lim_\mto \sum_{k=m}^\infty \frac{C_\Pnull }{k (\log k )^{1+\delta}} \\
    =\ &0.
  \end{align}
  Combining the above with \citet[Theorem 2]{waudby2024distribution}, we have that
  \begin{equation}
    \frac{1}{n}\sum_{i=1}^n (\xi_i^x)^2 \{ \mu^y - \widehat \mu_{i}^y(Z_i) \}^2 = \oPnullas \left ( 1/\log n \right ),
  \end{equation}
  completing the proof.
\end{proof}

\section{Auxiliary lemmas}

\begin{lemma}[Lipschitzness of $g^\star(\sqrt{\cdot})$]\label{lemma:Lipschitzness of gsqrt}
  Consider the function $g^\star$ given by
  \begin{equation}
    g^\star(x) = x^2 + 2 \log \Phi(x);\quad x \geq 0.
  \end{equation}
  Let $a > 0$. Then
  $g^\star(\sqrt{x})$ is Lipschitz continuous on $[a, \infty)$.
\end{lemma}
\begin{proof}
  Consider the derivative $f$ of $g^\star(\sqrt{\cdot})$:
  \begin{equation}
    f(x) := \frac{\dd g^\star(\sqrt{x})}{\dd x} = 1 + \frac{\phi(\sqrt{x})}{\sqrt{x} \Phi(\sqrt{x})}.
  \end{equation}
  Note that $f$ is continuous on $[a, \infty)$. Clearly, we have that $f(x) \to 1$ as $x \to \infty$ and that $f(a)$ is finite. It follows that $f$ is bounded on $[a ,\infty)$ by successive applications of the extreme value theorem and hence $g^\star(\sqrt{\cdot})$ is Lipschitz on $[a, \infty)$.
\end{proof}

\begin{lemma}[Uniform multiplicative ``continuity'' of $g^\star$]\label{lemma:g-properties}
  Consider the function $g^\star : [0,\infty) \to \RR$ defined by
  \begin{equation}
    g^\star(x) = x^2 + 2\log \Phi(x);~~x \geq 0.
  \end{equation}
  Fix $\delta \in (0, 1)$. For any $x \geq 0$, it holds that
  \begin{equation}
    \left \lvert g^\star \left ( \frac{x}{1-\delta} \right ) - \frac{g^\star \left ( x \right )}{(1-\delta)^2} \right \rvert \leq -g^\star(0) \left (  \frac{1}{(1-\delta)^2} - 1 \right ).
  \end{equation}
\end{lemma}
\begin{proof}
  Define the function
  \begin{equation}
    f(x) := g^\star \left ( \frac{x}{1-\delta} \right ) - \frac{g^\star(x)}{(1-\delta)^2} = 2 \log \Phi\left (\frac{x}{1-\delta}\right ) - \frac{2 \log \Phi(x)}{(1-\delta)^2}
  \end{equation}
  for $x \geq 0$. Computing the derivative of $f$, we have
  \begin{equation}
    \frac{\dd f(x)}{\dd x} = \frac{2}{1-\delta} \left ( \frac{\phi \left ( \frac{x}{1-\delta} \right )}{\Phi \left ( \frac{x}{1-\delta} \right )} - \frac{\phi(x)}{\Phi(x)(1-\delta)} \right ).
  \end{equation}
  Notice that for $x \geq 0$, it holds that
  \begin{equation}
    \frac{\phi\left (\frac{x}{1-\delta}\right )}{\Phi \left ( \frac{x}{1-\delta}\right )} \leq \frac{\phi(x)}{\Phi(x)}\leq \frac{\phi(x)}{\Phi(x) (1-\delta)},
  \end{equation}
  where the first inequality follows from the fact that $\phi$ and $\Phi$ are decreasing and increasing on $[0, \infty)$, respectively, and from the fact that $0 < \delta < 1$. Consequently, the derivative of $f$ is negative on $[0,\infty)$.
  Combined with the facts that $f(x) \to 0$ as $\xto$ and that $f(0) > 0$, it is the case that $f$ is maximized at $x = 0$. In other words, $|f(x)| \leq f(0)$ for all $x \geq 0$, and hence for all $x \geq 0$,
  \begin{equation}
    |f(x)| = g^\star \left ( \frac{x}{1-\delta} \right ) - \frac{g^\star(x)}{(1-\delta)^2} \leq -g^\star(0) \left ( \frac{1}{(1-\delta)^2} - 1 \right ),
  \end{equation}
  with equality attained at 0.
  This completes the proof.
\end{proof}

\begin{lemma}\label{lemma:uniform slutsky}
  Let $X$ be a random variable and let $f$ be a continuous, strictly increasing function on $[a, \infty)$ for some $a \in \RR$. Then
  \begin{equation}
    \lim_{h \to 0^+}\sup_{x \geq a}\left \lvert \P \left ( f \left ( \frac{X}{1+h} \right ) \geq x \right ) - \P \left ( f \left ( X \right ) \geq x \right )\right \rvert = 0.
  \end{equation}
\end{lemma}
\begin{proof}
  Using continuity and monotonicity (and hence invertibility) of $f$, we have for any $x \geq a$,
  \begin{align}
    \P \left ( f \left ( \frac{X}{1+h} \right ) \geq x \right ) &= \P \left ( \frac{X}{1+h} \geq f^{-1}(x) \right ).
  \end{align}
  By Slutsky's theorem, we have for any $y$ in the preimage $f^{-1}([a, \infty))$,
  \begin{equation}
    \lim_{h \to 0^+}\P \left ( \frac{X}{1+h} \geq y \right ) = \P \left ( X \geq y \right ).
  \end{equation}
  Using the above and the fact that $f$ is invertible on $[a, \infty)$, we have for any $x$,
  \begin{align}
    \lim_{h \to 0^+} \left \lvert \P \left ( f \left ( \frac{X}{1+h} \right )\geq x  \right ) - \P \left ( f(X) \geq x \right )\right \rvert &= \lim_{h \to 0^+} \left \lvert \P \left ( \frac{X}{1+h}\geq f^{-1}(x)  \right ) - \P \left ( X \geq f^{-1}(x) \right )\right \rvert\\
    &= 0.
  \end{align}
  By \cref{lemma:quantile uniformity for free}, we have that the above holds quantile-uniformly:
  \begin{equation}
    \lim_{h \to 0^+} \sup_{x \geq a }\left \lvert \P \left ( f \left ( \frac{X}{1+h} \right )\geq x  \right ) - \P \left ( f(X) \geq x \right )\right \rvert,
  \end{equation}
  completing the proof.
\end{proof}
The following serves as a generalization of \citet[Lemma 2.11]{van2000asymptotic} for CDF-like functions.

\begin{lemma}[Quantile-uniformity for continuous CDF-like functions]\label{lemma:quantile uniformity for free}
  Let $\infseqn{Y_n}$ be a sequence of random variables on a probability space $\probspace$. Let $a \in \RR$. Suppose that for some non-decreasing and continuous function $F : [a, \infty) \to (0, 1)$ satisfying $\lim_{x\to\infty} F(x) = 1$, it holds that
  \begin{equation}
    \forall x \geq a,\quad \lim_\mto \abs{\P \left ( Y_m \geq x \right ) - (1-F(x))} = 0.
  \end{equation}
  Then the same convergence occurs uniformly in $x \in [a, \infty)$, i.e.
  \begin{equation}
    \lim_\mto \sup_{x\geq a} \abs{\P \left ( Y_m \geq x \right ) - (1-F(x))} = 0.
  \end{equation}
\end{lemma}
\begin{proof}
  Fix $L \in \NN \setminus \{1\}$. By continuity of $F$, choose $x_0 = a$ and $x_1, \dots, x_{L-1}$ so that $F(x_\ell) = F(a) + (1-F(a)) \ell/L$ for each $\ell \in [L-1]$. Take $x_L = \infty$ so that $F(x_L) := \lim_{x \to \infty} F(x) = 1$ as a convention, noting that $F(x_L) = 1 = F(a) + (1-F(a))L/L $. Then, for any $x \in [a, \infty)$, we have that whenever $x_\ell \leq x < x_{\ell+1}$,
  \begin{align}
    \P \left ( Y_m \geq x \right ) - (1-F(x)) &\leq \P \left ( Y_m \geq x_\ell \right ) - (1-F(x_{\ell + 1})) \\
                                                           &=  \P \left ( Y_m \geq x_\ell \right ) - \left [1- \left (F(a) + (1 - F(a))\frac{\ell + 1}{L} \right ) \right ] \\
                                                           &= \P \left ( Y_m \geq x_\ell \right ) - (1 - F(x_\ell)) + \frac{1 - F(a)}{L} \\
    &\leq \P \left ( Y_m \geq x_\ell \right ) - (1 - F(x_\ell)) + \frac{1}{L}
  \end{align}
  Similarly, it holds that for any $x \in [x_\ell, x_{\ell + 1})$,
  \begin{align}
    \P \left ( Y_m \geq x \right ) - (1-F(x)) &\geq \P( Y_m \geq x_{\ell+1}) - (1-F(x_\ell)) \\
                                              &= \P \left ( Y_m \geq x_{\ell+1} \right ) - \left [ 1 - \left ( F(a) + (1-F(a)) \frac{\ell}{L} \right ) \right ] \\
                                              &= \P \left ( Y_m \geq x_{\ell+1} \right ) - \left [ 1 - \left ( F(a) + (1-F(a)) \frac{\ell + 1}{L} - (1-F(a)) \frac{1}{L} \right ) \right ] \\
                                              &= \P \left ( Y_m \geq x_{\ell+1} \right ) - \left [ 1 - F(x_{\ell+1})  \right ] - \frac{1-F(a)}{L} \\
                                              &\geq \P \left ( Y_m \geq x_{\ell+1} \right ) - \left [ 1 - F(x_{\ell+1})  \right ] - \frac{1}{L}.
  \end{align}
  Combining the above to inequalities, taking suprema over $x \geq a$ and limits as $\mto$, we have that
  \begin{equation}
    \lim_\mto \sup_{x \geq a} \abs{\P \left ( Y_m \geq x \right ) - (1-F(x))} \leq \frac{1}{L}.
  \end{equation}
  Since $L$ was arbitrary, we have that
  \begin{equation}
    \lim_\mto \sup_{x \geq a} \abs{\P \left ( Y_m \geq x \right ) - (1-F(x))} = 0,
  \end{equation}
  completing the proof.
\end{proof}

\end{document}